\documentclass[12p]{amsart}
\usepackage[margin=0.7in]{geometry}
\usepackage{version,amsmath, amssymb, mathtools}
\usepackage{color}
\usepackage{amsthm, amsopn, amsfonts, xypic}
\usepackage[colorlinks=true]{hyperref}
\usepackage{mathrsfs}
\usepackage{slashed}
\hypersetup{urlcolor=blue, citecolor=blue}

\usepackage[color=yellow]{todonotes}

\let\arXiv\arxiv
\def\doi#1{ {\href{http://dx.doi.org/#1}
   {{\mdseries\ttfamily DOI}}}}

\newcommand{\beq}{\begin{equation}}\newcommand{\eq}{\end{equation}}

\renewcommand{\tt}{ \tilde t}

\newcommand{\B}{{\underline{L}}}

\newcommand{\tpa}{{\overline{\pa}}}

\newcommand{\pa}{{\partial}}
\newcommand{\tphi}{{\tilde{\phi}}}

\newcommand{\rs}{{\tilde r}}

\newtheorem{theorem}{Theorem}

\newtheorem{lemma}[theorem]{Lemma}
\newtheorem{corr}[theorem]{Corollary}

\newtheorem{proposition}[theorem]{Proposition}

\newcommand{\R}{{\mathbb R}}

\newcommand{\ang}{{\not\negmedspace\partial }}

\newcommand{\la}{\langle}
\newcommand{\ra}{\rangle}

\renewcommand{\S}{{\mathbb S}}
\newcommand{\M}{\mathcal M}

\begin{document}

\title
{
 The weak null condition on Kerr backgrounds
}

\author{Hans Lindblad}
\address{Department of Mathematics, Johns Hopkins University, Baltimore,
  MD  21218}
\email{lindblad@math.jhu.edu}
\author{Mihai Tohaneanu}
\address{Department of Mathematics, University of Kentucky, Lexington,
  KY 40506}
\email{mihai.tohaneanu@uky.edu}

\begin{abstract}

We study a system of semilinear wave equations on Kerr backgrounds that satisfies the weak null condition. Under the assumption of small initial data, we prove global existence and pointwise decay estimates.
 \end{abstract}

\mathtoolsset{showonlyrefs=true}

\maketitle

\tableofcontents

\emph{\emph{}}

\section{Introduction}

The semilinear system of wave equations in $\R^{1+3}$
\begin{equation}
\Box \phi = Q[\pa\phi,\pa\phi], \qquad \phi |_{t=0} = \phi_0, \qquad \pa_t \phi |_{t=0} = \phi_1,
\end{equation}
where $Q$ is a quadratic form,
for small initial data has been studied extensively. For the scalar equation, it is known that the solution can blows up in finite time for $\Box \phi = (\pa_t\phi)^2$, see \cite{J}. On the other hand, if the nonlinearity satisfies the null condition by Klainerman \cite{K1}, e.g.  $\Box \phi = (\pa_t\phi)^2-|\pa_x \phi|^2$, it was shown independently in \cite{Ch} and \cite{K2} that the solution exists globally. This result was extended to quasilinear systems with multiple speeds, as well as the case of exterior domains; see, for instance, \cite{MNS}, \cite{MS1}, \cite{MS2}, \cite{H}, \cite{LNS}, \cite{KS}, \cite{Al}, \cite{L0,L1}, \cite{ST}, \cite{FM}. There have also been many works for small data in the variable coefficient case. Almost global existence for nontrapping metrics was shown in \cite{BH}, \cite{SW}. Global existence for stationary, small perturbations of Minkowski was shown in \cite{WY}, for nonstationary, compactly supported perturbations in \cite{Y}, and for large, asymptotically flat perturbations that satisfy the strong local energy decay estimates in \cite{LooiT}. In the context of black holes, global existence was shown in \cite{Luk} for Kerr space-times with small angular momentum, and in \cite{AAG} for the Reissner-Nordstr\"om backgrounds.

Written in harmonic coordinates, the Einstein Equations take the form
\beq\label{MinkEins}
  \Box_{g} g_{\mu\nu} =P[\pa_\mu g,\pa_\nu g]+Q_{\mu\nu}[\pa g,\pa g],
\eq
where $\Box_g$ is the wave operator on the background of the Lorentzian metric $g$, and $P$ and $Q_{\mu\nu}$ are quadratic forms with coefficients depending on the metric. Unfortunately the nonlinear terms do not satisfy the null condition. Yet Christodoulou-Klainerman\cite{CK2} were able to prove global existence for Einstein vacuum equations ${R}_{\mu\nu}\!\equiv\!0$ for small asymptotically flat
initial
data.
Their proof avoids using coordinates since it
 was believed the metric in harmonic coordinates would blow up for large times.
 However, later Lindblad-Rodnianski \cite{LR1} noticed that Einstein's equations in harmonic coordinates satisfy a weak null condition, and subsequently used it to prove stability of Minkowski in harmonic coordinates \cite{LR2}, \cite{LR3}. Whereas it is still unknown if general equations satisfying the weak null condition have global existence for small initial data, there has been a number of results in that direction, including detailed asymptotics of the solution, see for example \cite{Al}, \cite{L0,L1,L2}, \cite{K18}, \cite{DP}, \cite{Yu1}, \cite{Yu2}.

There has recently been a lot of activity in proving asymptotic stability of black holes. As a
first step people have proved decay of solutions to wave equations
on Schwarzschild and Kerr background \cite{BS1,BS2,BSter,MMTT,DR1} and \cite{TT,DRS16,AB}. People have also studied semilinear perturbations
\cite{Luk,IK15} satisfying the null condition, but apart from our
recent papers \cite{LT1,LT2}, little is known about quasilinear
perturbations or semilinear perturbations satisfying the weak null condition. There has more recently been  progress on the nonlinear stability of Schwarzschild and Kerr \cite{KS1,KS2,KS3,DHRT21,GKS2}. 
These proofs are very long, using sophisticated geometric constructions. We hope that studying models of Einstein's equations in wave coordinates will simplify the proofs and lead to a better understanding and extensions as it did for the stability of Minkowski space.

Finally we remark that there are several recent works on the cosmological case. Hintz-Vasy proved the stability of Kerr de Sitter with small angular momentum \cite{HV16}, see also  \cite{F21,F22} for an alternative proof. More recently there have been works on the wave equation on Kerr-deSitter background for large angular momentum assuming there are no growing modes \cite{LV22,M22}.

\subsubsection{The semilinear Einstein model}
An example of a simple semilinear systems satisfying the weak null condition, but not the classical null condition, is the  system
\[
\Box \phi_1 = (\pa_t \phi_2)^2, \qquad \Box\phi_2 = 0
\]
It is trivial to see that this has global solutions, and moreover that $\phi_1$ decays slower than $1/t$. A less trivial  example is the semilinear system
\[\label{eq:decoupledsystem}
\Box \phi_1 = (\pa_t \phi_2)^2 + Q_1[\pa\phi, \pa\phi], \qquad \Box\phi_2 = Q_2[\pa\phi, \pa\phi]
\]
where $Q_j$ are null forms. These systems have the advantage the components $\phi_1$ and $\phi_2$ decouple to highest order. For Einstein's equations there is the additional difficulty that this decoupling can only be seen in a null frame, and contractions with the frame do not commute with the wave operator as far as the $L^2$ estimate. Hence a more realistic model is the system is
\beq
\Box \phi_{\mu\nu} = P[\pa_\mu \phi,\pa_\nu \phi]+Q_{\mu\nu}[\pa \phi,\pa\phi],
\eq
where $P$ is assumed to have a certain weak null structure.
Contracting with a nullframe this resembles the decoupled systems with $\phi_{\underline{L}\underline{L}}$ in place of $\phi_1$, where $\underline{L}^\mu\pa_\mu =\pa_t-\pa_r$, and $\phi_2$ replaced by the other components $\phi_{TU}$ where $T$ is tangential to the outgoing light cones.
The only really bad component is $\pa \phi_{\underline{L}\underline{L}}$ but this one does not show up quadratically in $P$ for Einstein's Equations. It shows up linearly but multiplied with a component $\pa \phi_{LL} $ that has vanishing radiation field due to the wave coordinate condition.

With the goal of understanding Einstein's Equations in (generalized) harmonic coordinates close to Kerr with small angular momentum, we will focus on the following system, which resembles the semilinear part of Einstein's equations:
\beq\label{eq:Einsteinsemilinearmodel}
\Box_{K} \phi_{\mu\nu}=P[\pa_\mu \phi,\pa_\nu \phi]+Q_{\mu\nu}[\pa \phi,\pa\phi], \qquad \tt\geq 0,\qquad \phi |_{\tt=0} = \phi_0, \qquad \tilde T \phi|_{\tt=0} = \phi_1.
\end{equation}

Here $\Box_K$ denotes the d'Alembertian with respect to the Kerr metric, and $\tilde T$ is a smooth, everywhere timelike vector field that equals
$\partial_t$ away from the black hole.  The coordinate $\tt$
is chosen so that the slice $\tt=const$ are space-like and $\tt=t$
away from the black hole. For simplicity we will consider compactly supported smooth initial data, but suitably weighted Sobolev spaces of large enough order would suffice.
Moreover, $Q_{\mu\nu}$ are null forms and $P$ is a symmetric quadratic form:
\begin{equation}
P[\phi,\psi]=P^{\alpha\beta\gamma\delta}(x/ \tt) \, \phi_{\alpha\beta}\psi_{\gamma\delta},
\end{equation}
with coefficients with a certain weak null structure.
 We remove the component $\pa \phi_{\underline{L}\underline{L}}$ by imposing the condition
\beq
P^{\underline{L}\underline{L}\alpha\beta}(x/ \tt) =P^{\alpha\beta\underline{L}\underline{L}}(x/ \tt) =0.
\eq

For this system we cannot have different
energy estimates for different components because the null structure is only seen in a null frame and contractions with the frame do not commute with the wave operator.  Because of this one can not get the decay estimates directly from the $L^2$ estimates but one has to use the equations again to get improved decay estimates.
 As a result, the proof is more involved.
The method we develop avoids boosts vector fields and combines local energy decay at the origin with estimates in characteristic coordinates at the light cone. It gives an essentially optimal decay of almost $\tt^{-1}$, which is an improvement over $\tt^{-1/2}$ which can be obtained more easily from energy estimates. The method in particular works close to  Minkowski  where it gives the optimal decay without using boosts.

Finally we remark that this system can be combined with the quasilinear system that we previously studied \cite{LT1}, \cite{LT2} (see also \cite{Looi} for improved pointwise bounds) to resemble also the quasilinear part of Einstein's equations
\beq\label{eq:Einsteinquasilinearmodel}
\Box_{g[\phi]} \phi_{\mu\nu}=P[\pa_\mu \phi,\pa_\nu \phi]+Q_{\mu\nu}[\pa \phi,\pa\phi],
\eq
where
$$
g^{\alpha\beta}[\phi]=K^{\alpha\beta}+H^{\alpha\beta}[\phi], \qquad\text{where}\quad
H^{\alpha\beta}[\phi]=H^{\alpha\beta\mu\nu}(x/ \tt) \, \phi_{\mu\nu},\quad  \text{and}\quad H^{\underline{L}\underline{L}\mu\nu}(x/\tt) =0.
$$

\subsubsection{Statement of the results}
We are now ready to state the our main result. We define $\rs$ to be some function that equals $r$ near the event horizon, and $r_K^*$ away from it, see Section 2 for more details.

\begin{theorem}\label{mainthm}

Fix $R_0>r_e$, and assume that $\phi_0$, $\phi_1$ are smooth and compactly supported in $\rs\leq R_0$. Then there exists a global classical solution to \eqref{eq:Einsteinsemilinearmodel}, provided that, for a certain $\epsilon_0 \ll 1$ and large enough $N$, we have
\[
\mathcal{E}_{N}(0)=\|(\phi_{0}, \phi_1)\|_{H^{N+1}\times H^N} \leq \epsilon_0.
\]
Moreover, for some fixed positive integer $m$, independent of $N$, we have for any $\delta>0$
\begin{equation}\label{ptwseapbdsbadLEintro}
|\phi_{\leq N-m}| \lesssim \frac{\mathcal{E}_{N}(0)}{\la \tt-\rs\ra^{\delta}\la \tt\ra^{1-\delta}}, \qquad |\partial \phi_{\leq N-m}| \lesssim \frac{\tt^{\delta}\mathcal{E}_{N}(0)}{r\la \tt-\rs\ra^{1+\delta}},
\end{equation}
\beq\label{ptwseapbdsgoodLEintro}
|(\partial \phi_{TU})_{\leq N-m}| \lesssim \frac{\mathcal{E}_{N}(0)}{r \la \tt-\rs\ra^{1-\delta}}. 
\eq

\end{theorem} Note that is an improvement of the decay estimates we previously proved essentially by a factor of $\tt^{-1/2}$.
Note also the structure here, that a derivative decreases the homogeneity, but because the homogeneous vector fields we can use together with the wave operator do not span the tangent space at the origin or at the light cone a derivative only improves by a power of $r$ close to the origin and a power of $\tt-\rs$ close to the light cone. Note also that close to the light cone we have a better estimate for the good components which is due to the weak null structure.

\subsubsection{Structure of the proof} The starting point is the local energy estimate in  Section \ref{sec:localenergy}. The local energy scales like the energy which is consistent with a decay $\tt^{-1/2}$ of order $-1/2$ for $\phi$ and $-3/2$ for the derivatives, and this is also the decay we were able to obtain in our previous paper from a bound of the local energy applied to scaling and rotation vector fields, see Section \ref{sec:weakdecay}. Assuming these decay estimates one can go back into the equation and get improved decay estimates. In fact from these decay estimates the total decay of the inhomogeneous term would be $-3$ which would be consistent with a solution of the wave equation with decay of order $-1$. We prove this using $L^\infty$ estimates for the wave operator from Section \ref{sec:sharpdecay}. However the first improved estimates we obtain have the improved decay in $r$ or $\tt-\rs$ and we need improved decay in $\tt$. For this we have other estimates turn decay in $r$ or $\tt-\rs$ into decay in $\tt$, see Section \ref{sec:improvedbounds}. The whole argument is put together in the last section.

The paper is structured as follows. In Section 2 we introduce the Kerr metric, the vector fields we will use, and the local energy estimates which will play a key role in the proof. Sections 3, 4 and 5, and 6 contain various estimates that will allow us to extract the necessary pointwise bounds for (vector fields applied to) the solution. Finally, Section 7 contains the bootstrap argument.

\section{The Kerr metric and local energy estimates}\label{sec:localenergy}

\subsection{The Kerr metric}\label{sec:thecoordinates}\hfill

 The Kerr geometry in Boyer-Lindquist coordinates is given by
\[
ds^2 = g^K_{tt}dt^2 + g_{t\phi}dtd\phi + g^K_{rr}dr^2 + g^K_{\phi\phi}d\phi^2
 + g^K_{\theta\theta}d\theta^2,
\]
 where $t \in \R$, $r > 0$, $(\phi,\theta)$ are the spherical coordinates
on $\S^2$ and
\[
 g^K_{tt}=-\frac{\Delta-a^2\sin^2\theta}{\rho^2}, \qquad
 g^K_{t\phi}=-2a\frac{2Mr\sin^2\theta}{\rho^2}, \qquad
 g^K_{rr}=\frac{\rho^2}{\Delta},
 \]
\[ g^K_{\phi\phi}=\frac{(r^2+a^2)^2-a^2\Delta
\sin^2\theta}{\rho^2}\sin^2\theta, \qquad g^K_{\theta\theta}={\rho^2},
\]
with
\[
\Delta=r^2-2Mr+a^2, \qquad \rho^2=r^2+a^2\cos^2\theta.
\]

 Here $M$ represents the mass of the black hole, and $aM$ its angular momentum.

 A straightforward computation gives us the inverse of the metric:
\[ g_K^{tt}=-\frac{(r^2+a^2)^2-a^2\Delta\sin^2\theta}{\rho^2\Delta},
\qquad g_K^{t\phi}=-a\frac{2Mr}{\rho^2\Delta}, \qquad
g_K^{rr}=\frac{\Delta}{\rho^2},
\]
\[ g_K^{\phi\phi}=\frac{\Delta-a^2\sin^2\theta}{\rho^2\Delta\sin^2\theta}
, \qquad g_K^{\theta\theta}=\frac{1}{\rho^2}.
\]

The case $a = 0$ corresponds to the Schwarzschild space-time.  We shall
subsequently assume that $a$ is small $0 < a \ll M$, so that the Kerr
metric is a small perturbation of the Schwarzschild metric. Note also that the coefficients depend only $r$ and $\theta$ but are independent of
$\phi$ and $t$. We denote the d'Alembertian associated to the Kerr metric by $\Box_K$.

In the above coordinates the Kerr metric has singularities at $r = 0$,
on the equator $\theta = \pi/2$, and at the roots of $\Delta$, namely
$r_{\pm}=M\pm\sqrt{M^2-a^2}$. To remove the singularities at $r = r_{\pm}$ we
introduce functions $r_K^*=r_K^*(r)$, $v_{+}=t+r_K^*$ and $\phi_{+}=\phi_{+}(\phi,r)$ so that (see
\cite{HE})
\[
 dr_K^*=(r^2+a^2)\Delta^{-1}dr,
\qquad
 dv_{+}=dt+dr_K^*,
\qquad
 d\phi_{+}=d\phi+a\Delta^{-1}dr.
\]

Note that when $a=0$ the $r_K^*$ coordinate becomes the Schwarzschild Regge-Wheeler coordinate
\[
r^*=r+2M\log(r-2M)
\]

The Kerr metric can be written in the new coordinates $(v_+, r, \phi_+, \theta)$
\[
\begin{split}
ds^2= &\
-(1-\frac{2Mr}{\rho^2})dv_{+}^2+2drdv_{+}-4a\rho^{-2}Mr\sin^2\theta
dv_{+}d\phi_{+} -2a\sin^2\theta dr d\phi_{+} +\rho^2 d\theta^2 \\
& \ +\rho^{-2}[(r^2+a^2)^2-\Delta a^2\sin^2\theta]\sin^2\theta \,
d\phi_{+}^2
\end{split}
\]
which is smooth and nondegenerate across the event horizon up to but not including
$r = 0$. We introduce the function
\[
\tt = v_{+} - \mu(r),
\]
where $\mu$ is a smooth function of $r$. In the $(\tt,r,\phi_{+},
\theta)$ coordinates the metric has the form
\[
\begin{split}
ds^2= &\ (1-\frac{2Mr}{\rho^2}) d\tt^2
+2\left(1-(1-\frac{2Mr}{\rho^2})\mu'(r)\right) d\tt dr \\
 &\ -4a\rho^{-2}Mr\sin^2\theta d\tt d\phi_{+} + \Bigl(2 \mu'(r) -
 (1-\frac{2Mr}{\rho^2}) (\mu'(r))^2\Bigr)  dr^2 \\
 &\ -2a (1+2\rho^{-2}Mr\mu' (r))\sin^2 \theta dr d\phi_{+} +\rho^2
 d\theta^2 \\
 &\ +\rho^{-2}[(r^2+a^2)^2-\Delta a^2\sin^2\theta]\sin^2\theta
d\phi_{+}^2.
\end{split}
\]

On the function $\mu$ we impose the following two conditions:

(i) $\mu (r) \geq  r_K^*$ for $r > 2M$, with equality for $r >
{5M}/2$.

(ii)  The surfaces $\tt = const$ are space-like, i.e.
\[
\mu'(r) > 0, \qquad 2 - (1-\frac{2Mr}{\rho^2}) \mu'(r) > 0.
\]
As long as $a$ is small, we can use the same
function $\mu$ as in the case of the Schwarzschild space-time in \cite{MMTT}.

We also introduce
\[
\tphi = \zeta(r)\phi_{+}+(1-\zeta(r))\phi,
\]
where $\zeta$ is a cutoff function supported near the event horizon.

We fix $r_e$ satisfying $r_-<r_e<r_+$. The choice of $r_e$ is unimportant,
and for convenience we may simply use $r_e=M$ for all Kerr metrics
with $a/M\ll 1$. Let $\M =  \{ \tt \geq 0, \ r \geq r_e \}$, $ \Sigma(T) =  \M \cap \{ \tt = T \}$, and $ d\Sigma_K $ be the induced volume element on $\Sigma(T)$.

 Let $\tilde r$ denote a smooth strictly increasing function (of $r$) that equals $r$ for $r\leq R$ and $r_K^*$ for $r\geq 2R$ for some large $R$. We will use the coordinates $(\tt, x^{i})$, where $x^i =  \rs\omega$. Note that, since $r\approx \rs$, we can use $r^k$ and $\rs^k$ interchangeably when defining our spaces of functions in what follows.

\subsection{Vector fields and spaces of functions}\label{sec:vf}\hfill

 Our favorite sets of vector fields will be
\[
\partial = \{ \partial_{\tt}, \partial_i\}, \qquad \Omega = \{x^i \partial_j -
x^j \partial_i\}, \qquad S = \tt \partial_{\tt} + \tilde r \partial_{\tilde r},
\]
namely the generators of translations, rotations and scaling. We set
$Z = \{ \partial,\Omega,S\}$.

We also denote by $\ang$ the angular derivatives,
\[
\pa_j = \frac{x^i}{\rs} \pa_{\tilde r} + \ang_i
\]
and let
\[
\overline{\pa} := (\pa_v, \ang), \quad \pa_v = \pa_{\tt} + \pa_{\tilde r}
\]
denote the tangential derivatives. We also let $\underline{L} = \pa_{\tt}-\pa_{\rs}$.

For a triplet $\alpha=(i, j, k)$ we define $|\alpha| = i + 3j + 3k$ and
\[
u_{\alpha} = \pa^i \Omega^j S^k u, \quad u_{\leq m} =
(u_{\Lambda})_{|\Lambda|\le m}
\]

Given a norm $\|\cdot\|_X$, we write
\[
\| u_{\leq m}\|_X = \sum_{|\Lambda|\le m} \| u_{\Lambda}\|_X
\]

We define the classes $S^Z(r^k)$ of
functions in $\R^+ \times \R^3$ by
\[
 f \in S^Z(r^k) \Longleftrightarrow
|Z^j f(t, x)| \leq c_{j} \la r\ra^{k}, \quad j \geq 0.
\]

Given a family of functions $\mathcal{G}$, we will also use the notation
 \[
 f \in S^Z(r^k) \mathcal{G}
 \]
 to mean that
 \[
 f = \sum h_i g_i, \quad h_i\in S^Z(r^k), \quad g_i\in \mathcal{G}.
 \]

We will also use the notation $U$ for an element of $S^Z(1) Z$, and $T$ for an element of $S^Z(1) \tpa$.

An important observation is that, since
\[
\pa_v = \frac{\tt-\rs}{\tt}\pa_{\rs} + \frac{1}{\tt} S, \quad \ang\phi \in S^Z(r^{-1}) \Omega\phi
\]
we have
\beq\label{tgimp}
 |\tpa w| \lesssim \frac{\tt-\rs}{r} |\pa w| + \frac{1}{r} |\Omega w|. 
\eq

Moreover, an easy computation gives
\[
[\Box_K, \pa] \phi\in S^Z(r^{-2})\pa\pa^{\leq 1} \phi, \quad [\Box_K, \Omega]\phi \in S^Z(r^{-2})\pa\pa^{\leq 1} \phi,
\]
\[
[\Box_K, S] \phi \in S^Z(1) \Box_K \phi+ S^Z(r^{-2+})\pa \phi + S^Z(r^{-2+})\pa \Omega \phi + S^Z(r^{-2})\pa\pa^{\leq 1} \phi,
\]
and thus by induction we obtain that
\beq\label{comm}
[\Box_K, Z^{\alpha}] \phi = F_1+F_2, \quad F_1\in S^Z(1) (\Box_K \phi)_{\leq |\alpha|}, \quad F_2 \in S^Z(r^{-2+})\pa \phi_{\leq |\alpha|}.
\eq

We now claim that
\beq\label{hghnull}
[Z, \tpa] \in S^Z(1) \tpa + S^Z(r^{-1}) \pa
\eq

Indeed, we compute
\[
[\pa_{\tt}, \tpa] = 0
\]
\[
[\pa_i, \pa_v] = [\ang_i, \pa_{\tilde r}] \in S^Z(r^{-1}) \ang
\]
\[
[\pa_i, \ang] \in S^Z(r^{-1}) \pa
\]
\[
[\Omega, \pa_v] = 0
\]
\[
[\Omega, \ang] \in S^Z(1) \ang
\]
\[
[S, \pa_v] = \pa_v
\]
\[
[S, \ang] \in S^Z(1) \ang
\]

This proves \eqref{hghnull}.

Given vector fields $X$ and $Y$, we define
\[
\phi_{XY}= X^{\alpha}Y^{\beta}\phi_{\alpha\beta}
\]

Similarly, we can write the coefficients $P$ with respect to the vector frame $\{\B, \tpa\}$ as
\[
P^{\alpha\beta\gamma\delta} = P^{\B\B\gamma\delta}\B^{\alpha}\B^{\beta} + \sum P^{TU\gamma\delta} T^{\alpha} U^{\beta}
\]
\[
P^{\alpha\beta\gamma\delta} = P^{\alpha\beta\B\B}\B^{\gamma}\B^{\delta} + \sum P^{\alpha\beta TU} T^{\gamma} U^{\delta}
\]

The assumptions on the coefficients $P^{\alpha\beta\gamma\delta}$ are the following:
\beq\label{Pass}
P^{\alpha\beta\gamma\delta} \in S^Z(1),
\eq
\beq\label{weaknull}
P^{\underline{L}\underline{L}\alpha\beta}=P^{\alpha\beta\underline{L}\underline{L}}=0.
\eq

\eqref{weaknull} means that terms like $\B\phi_{\B\B} \pa\phi$ do not appear on the right hand side of \eqref{eq:Einsteinsemilinearmodel}.

The assumption on the null forms $Q_{\mu\nu}$ is that
\beq\label{Qass}
Q_{\mu\nu}[\pa\phi, \pa\phi] \in S^Z(1) \pa\phi \tpa\phi.
\eq

\subsection{Local energy estimates}\label{sec:LEE}\hfill

We consider a partition
of $ \R^{3}$ into the dyadic sets $A_R= \{R\leq \la \rs\ra  \leq 2R\}$ for
$R \geq 1$.

We now introduce the local energy norm $LE$
\begin{equation}
\begin{split}
 \| u\|_{LE} &= \sup_R  \| \la r\ra^{-\frac12} u\|_{L^2 (\M\cap \R \times A_R)}  \\
 \| u\|_{LE[\tt_0, \tt_1]} &= \sup_R  \| \la r\ra^{-\frac12} u\|_{L^2 (\M \cap [\tt_0, \tt_1] \times A_R)},
\end{split}
\label{ledef}\end{equation}
its $H^1$ counterpart
\begin{equation}
\begin{split}
  \| u\|_{LE^1} &= \| \pa u\|_{LE} + \| \la r\ra^{-1} u\|_{LE} \\
 \| u\|_{LE^1[\tt_0, \tt_1]} &= \| \pa u\|_{LE[\tt_0, \tt_1]} + \| \la r\ra^{-1} u\|_{LE[\tt_0, \tt_1]},
\end{split}
\end{equation}
as well as the dual norm
\begin{equation}
\begin{split}
 \| f\|_{LE^*} &= \sum_R  \| \la r\ra^{\frac12} f\|_{L^2 (\M\cap \R\times A_R)} \\
 \| f\|_{LE^*[\tt_0, \tt_1]} &= \sum_R  \| \la r\ra^{\frac12} f\|_{L^2 (\M \cap [\tt_0, \tt_1] \times A_R)}.
\end{split}
\label{lesdef}\end{equation}

 We also define similar norms for higher Sobolev regularity
\[
\begin{split}
  \| u_{\leq m}\|_{LE^1} &= \sum_{|\alpha| \leq m} \| u_{\alpha}\|_{LE^1} \\
  \| u_{\leq m}\|_{LE^1[\tt_0, \tt_1]} &= \sum_{|\alpha| \leq m} \| u_{\alpha}\|_{LE^1[\tt_0, \tt_1]} \\
  \| u_{\leq m}\|_{LE[\tt_0, \tt_1]} &= \sum_{|\alpha| \leq m} \| u_{\alpha}\|_{LE[\tt_0, \tt_1]},
\end{split}
\]
respectively
\[
\begin{split}
  \| f\|_{LE^{*,k}} &=  \sum_{|\alpha| \leq k}  \| \partial^\alpha f\|_{LE^{*}} \\
  \| f\|_{LE^{*,k}[\tt_0, \tt_1]} &=  \sum_{|\alpha| \leq k}  \| \partial^\alpha f\|_{LE^{*}[\tt_0, \tt_1]}.
\end{split}
\]

Finally, we introduce a weaker version of the local energy
decay norm
\[
\begin{split}
  \| u\|_{LE^1_{w}} &= \| (1-\chi_{ps}) \pa u\|_{LE} + \| \pa_r u\|_{LE} + \| \la r\ra^{-1} u\|_{LE} \\
  \| u\|_{LE^1_{w}[\tt_0, \tt_1]} &= \| (1-\chi_{ps}) \pa u\|_{LE[\tt_0, \tt_1]}  + \| \pa_r u\|_{LE[\tt_0, \tt_1]} + \| \la r\ra^{-1} u\|_{LE[\tt_0, \tt_1]},
\end{split}
\]
To measure the inhomogeneous term, we define
\[
\begin{split}
 \| f\|_{LE^*_w} &= \inf_{f_1+f_2=f} \| f_1\|_{L^1 L^2}+ \| (1-\chi_{ps}) f_2\|_{LE^*} \\
\| f\|_{LE^*_w[\tt_0, \tt_1]} &= \inf_{f_1+f_2=f}  \| f_1\|_{L^1[\tt_0, \tt_1] L^2}+  \| (1-\chi_{ps}) f_2\|_{LE^*[\tt_0, \tt_1]}.
\end{split}
\]

Here $\chi_{ps}$  is a smooth, compactly
supported spatial cutoff function that equals $1$ in a neighborhood of the trapped set. We also define the higher order weak norms as above.

We define the (nondegenerate) energy
\[
E[u](\tt) = \left(\int_{\Sigma(\tt)} |\pa u|^2 d\Sigma_K\right)^{1/2}.
\]

We now fix some $\delta_1 \ll 1$, and define
\beq\label{energy}
\mathcal{E}_N(T) = \sup_{0\leq\tt\leq T} E[\phi_{\leq N}](\tt) + \|\phi_{\leq N}\|_{LE_w^1[0, T]} + \|\la \tt-\rs\ra^{\frac{-1-\delta_1}2} \tpa\phi_{\leq N}\|_{L^2[0, T] L^2(r\geq R_1)}.
\eq

We will need the following local energy estimates for the linear problem:
\begin{lemma}\label{hghorderLE}
Assume that $\Box_K\phi = F$, and $N$ is any nonnegative integer. We then have for any $T\geq 0$ that
\beq\label{wle}
\mathcal{E}_N(T) \lesssim  \mathcal{E}_N(0) + \|F_{\leq N}\|_{L^1[0, T]L^2+ LE^*_w[0, T]} 
\end{equation}
where the implicit constant is independent of $T$.

\end{lemma}
\begin{proof}
Indeed, Theorem 4.5 from \cite{TT} gives the desired bound for the first two terms. On the other hand, Lemma 4.3 in \cite{LT1} and Cauchy Schwarz yield
\beq\label{t-r}
\|\la \tt-\rs\ra^{\frac{-1-\delta_1}2} \tpa\phi\|_{L^2[0, T] L^2(r\geq R_1)} \lesssim \|\phi\|_{LE_w^1[0, T]} + \|F\|_{LE^*_w[0, T]},
\eq
which is the desired bound when $N=0$. Moreover, for any multiindex $\alpha$ we have from applying \eqref{t-r} to $\phi_{\alpha}$ that
\[
\|\la \tt-\rs\ra^{\frac{-1-\delta_1}2} \tpa\phi_{\alpha}\|_{L^2[0, T] L^2(r\geq R_1)} \lesssim \|\phi_{\alpha}\|_{LE_w^1[0, T]} + \|F_{\alpha}\|_{LE^*_w[0, T]} + \|[\Box_K, Z^{\alpha}] \phi\|_{LE^*_w[0, T]}. 
\]

We are left with bounding the last term on RHS. By \eqref{comm} we have
\[
\|[\Box_K, Z^{\alpha}] \phi\|_{LE^*_w[0, T]} \lesssim \|F_{\leq |\alpha|}\|_{LE^*_w[0, T]} + \|r^{-2+}\pa \phi_{\leq |\alpha|}\|_{LE^*_w[0, T]} \lesssim \|F_{\leq |\alpha|}\|_{LE^*_w[0, T]} + \|\phi_{\leq |\alpha|}\|_{LE_w^1[0, T]}.
\]

This finishes the proof of the lemma.
\end{proof}

The first estimate of this kind was obtained by Morawetz for the Klein-Gordon equation \cite{M}. In the Schwarzschild case, similar estimates were shown in \cite{BS1, BS2}, \cite{BSter}, \cite{DR1}, \cite{DR2}, \cite{MMTT}. The estimate for Kerr with small angular momentum was proven in \cite{TT} (see also \cite{AB} and \cite{DaRoNotes} for related works). For large angular momentum see \cite{DRS16} ($|a|<M$), and \cite{Ar} ($|a|=M$).

\section{Pointwise estimates from local energy decay estimates}\label{sec:weakdecay}

The goal of this section is to show how to extract (weak) pointwise estimates from local energy norms. These bounds will serve as the starting point in an iteration that will yield strong enough pointwise bounds to close the bootstrap argument in Section 7.

Let
 \[
C_{T} = \{ T \leq \tt \leq 2T, \ \ \rs \leq \tt\}.
\]

 We use a double dyadic decomposition of $C_T$
with respect to either the size of $\tt-\rs$ or the size of $r$, depending
on whether we are close or far from the cone,
\[
C_{T} ={\bigcup}_{1\leq  R \leq T/4}  C_{T}^{R} \, \,\,\bigcup\,\,\, {\bigcup}_{1\leq  U < T/4} C_T^{U},
\]
where for $R,U > 1$ we set
\[
 C_{T}^{R} = C_T \cap \{ R < r < 2R \},
\qquad
C_{T}^{U} = C_T \cap \{ U < \tt-\rs < 2U\},
\]
while for $R=1$ and $U= 1$ we have
\[
 C_{T}^{R=1} = C_T \cap \{ 0 < r < 2 \},
\qquad
C_{T}^{U=1} = C_T \cap \{ 0 < \tt-\rs < 2\}.
\]
The sets $C_T^R$ and $C_T^U$ represent the setting in which we apply
Sobolev embeddings, which allow us to obtain pointwise bounds from
$L^2$ bounds. Precisely, we have (see Lemma 3.8 from \cite{MTT} and Lemma 6.2 in \cite{LT1}):

\begin{lemma} \label{l:l2tolinf}
 For any function $w$ and all $T \geq 1$ and $1 \leq R,U \leq T/4$  we have
 \begin{equation}
  \!  \| w\|_{L^\infty(C_T^{R})} \lesssim
\frac{1}{T^{\frac12} R^{\frac32}} \sum_{i\leq 1, j \leq 2}
    \|S^i \Omega^j w\|_{L^2( C_T^{R})} +  \frac{1}{T^{\frac12} R^{\frac12}}
\sum_{i\leq 1, j \leq 2}  \|S^i \Omega^j \pa w\|_{L^2( C_T^{R})},
    \label{l2tolinf-r}\end{equation}
respectively
  \begin{equation}
    \| w\|_{L^\infty(C_T^{U})} \lesssim \frac{1}{T^{\frac32} U^{\frac12}} \sum_{i\leq 1, j \leq 2}
    \|S^i \Omega^j w\|_{L^2( C_T^{U})} +  \frac{U^{\frac12}}{T^{\frac32}}
 \sum_{i\leq 1, j \leq 2} \|S^i \Omega^j \pa w\|_{L^2(C_T^{U})}.
    \label{l2tolinf-u}\end{equation}
\end{lemma}

Using the lemma above, we prove the following pointwise bound:
\begin{equation}\label{ptdecayu}
\|w\|_{L^{\infty}(C_T)} \lesssim \la \tt\ra^{-1} \la \tt-\rs\ra^{1/2} \|w_{\leq 12}\|_{LE^1[T, 2T]}.
\end{equation}

Indeed, in the region $C_T^{R}$, this is an immediate application of \eqref{l2tolinf-r}. On the other hand, in the region $C_T^{U}$ this follows from \eqref{l2tolinf-u} and Hardy's inequality, see (6.7) in \cite{LT1}.

We also need an $L^{\infty}$ bound on the derivative that is better than \eqref{ptdecayu} for large $r$. This is the content of the following, which is essentially Proposition 3.5 in \cite{LooiT}
\begin{proposition}\label{DerProp}
Let
$$\mu := \min(\la \tt\ra,\la \tt-\rs\ra)^{1/2}.$$
Assume that $\phi$ solve \eqref{eq:Einsteinsemilinearmodel} for $t\in [T, 2T]$. Then for any dyadic region $C\in\{C_T^R, C_U^R\}$ and $m\geq 0$ we have 
\begin{equation}\label{deBound}
 \|\pa\phi_{\leq m}\|_{L^\infty(C)} \le \bar C_{m} \frac1\mu \left(\frac1{\la r\ra} + \|\pa\phi_{\le \frac{m+10}2} \|_{L^{\infty}(C)}\right)\|\phi_{\le m+5} \|_{LE^1[T,2T]}
\end{equation}
\end{proposition}

Here the crucial estimate was the following Klainerman-Sideris type estimate,
see Lemma 6.3 in \cite{LT1} (for Schwarzschild) combined with the remarks after (5.13) in \cite{LT2}:

\begin{lemma}\label{KS1}  For any $w$ we have in the region $r\geq 2R_1$ that
\begin{equation}\label{improvetwoderivs}
|\partial^2 w| \lesssim \frac{\tt}{r\la t-\rs\ra} \sum_{i+ j \leq 1} |\pa S^i \Omega^j w| + \frac{t}{\la t-\rs\ra} |\Box_K w|
\end{equation}
\end{lemma}

We now apply \eqref{l2tolinf-u} to $\pa \phi_{\Lambda}$. We obtain
\beq\label{dervcone} \begin{split}
\| \partial \phi_{\Lambda}\|_{L^\infty(C_T^{U})} \lesssim & \frac{1}{T^{\frac32} U^{\frac12}} \sum_{i\leq 1, j \leq 2}
    \|S^i \Omega^j \partial \phi_{\Lambda}\|_{L^2( C_T^{U})} +  \frac{U^{\frac12}}{T^{\frac32}}
\sum_{i\leq 1, j \leq 2}     \|S^i \Omega^j \partial^2 \phi_{\Lambda}\|_{L^2(C_T^{U})} \\ & \lesssim \frac{1}{TU^{\frac12}} \|\phi_{\leq |\Lambda|+13}\|_{LE^1[T, 2T]} +
 \frac{1}{(TU)^{\frac12}}  \|(\Box_K \phi)_{\leq |\Lambda|+10}\|_{L^2( C_T^{U})}
 \end{split}
\eq

Since
\[
|(\Box_K \phi)_{\leq |\Lambda|+10}| \lesssim |\pa \phi_{\leq \frac{|\Lambda|}2 + 5}| |\pa\phi_{\leq |\Lambda|+10}|
\]
the conclusion follows in the region $C_T^U$. A similar computation yields the result in $C_T^R$.

\section{Improved pointwise bounds}\label{sec:improvedbounds}

We will use three lemmas that will help us improve our pointwise bounds. The first one is Proposition 3.14 from \cite{MTT}, which will allow us to turn $r$-decay into $t$-decay in the region $r\leq t/2$.

\begin{lemma} \label{lv:smallr}
   The the following estimate holds for all $m\geq 0$ and some fixed ($m$-independent) $n$:
  \begin{equation}
    \| u_{\leq m} \|_{LE^1(C_{T}^{<T/2})} \lesssim
T^{-1} \| \la r \ra  u_{\leq m+n}\|_{LE^1( C_{T}^{<T/2})}
+   \|(\Box_K u)_{\leq m+n} \|_{LE^*( C_{T}^{<T/2})}.
    \label{smallr:main}
  \end{equation}
\end{lemma}

The second lemma is a slight modification of Lemma 3.11 from \cite{MTT}, the difference being that we may not enlarge our regions in time. The role of the lemma is to gain a factor of $\frac{\tt}{r\la \tt-\rs\ra}$ for the derivative.

We let $\tilde C_{T}^R$ and $\tilde C_{T}^U$ denote enlargements of $C_{T}^R$ and $C_{T}^U$ in space (but not in time) that contain all the integral curves of the scaling vector field $S$ (i.e. if $(t,x)\in C_T^{R}$ then $(st,sx)\in \tilde C_{T}^R$ as long as $T\leq st\leq 2T$ and similarly for $C_T^{U}$). More precisely, let
\[
\tilde C_{T}^R = \{T\leq \tt\leq 2T, \quad \frac{8}{10} \frac{T}{2R} \leq \frac{\tt}{\rs} \leq \frac{12}{10} \frac{2T}{R}\}, \quad \tilde C_{T}^R (\tau) = \tilde C_{T}^R \cap \{\tt=\tau\}.
\]
\[
\tilde C_{T}^U = \{T\leq \tt\leq 2T, \quad \frac{8}{10} \frac{T}{T-2U} \leq \frac{\tt}{\rs} \leq \frac{12}{10} \frac{2T}{2T-U}\}, \quad \tilde C_{T}^U (\tau) = \tilde C_{T}^U \cap \{\tt=\tau\}
\]

An important observation here is that $\rs\approx R$ and $\tt-\rs \approx U$ in $\tilde C_{T}^R$ and $\tilde C_{T}^U$ respectively.

\begin{lemma}\label{lem:L2derL2}
For $1 \ll U,R \leq T/4$ we have
\begin{equation}
  \| \pa w\|_{L^2(C_{T}^{R})} \lesssim R^{-1}
 \| w\|_{L^2(\tilde C_{T}^R)}
  + T^{-1}\big( \|S w \|_{L^2(\tilde C_{T}^R)}+ \|S^2 w \|_{L^2(\tilde C_{T}^R)}\big)
  + R \| \Box_K w\|_{L^2(\tilde C_{T}^R)}
\label{ctrdu}\end{equation}
respectively
\begin{equation}
  \| \pa w\|_{L^2(C_{T}^U)} \lesssim U^{-1}
 \big(\| w\|_{L^2(\tilde C_{T}^U)}
  + \|S w\|_{L^2(\tilde C_{T}^U)}+ \|S^2 w\|_{L^2(\tilde C_{T}^U)}\big)
  + T \| \Box_K w\|_{L^2(\tilde C_{T}^U)}
\label{ctudu}\end{equation} \label{l2du}
\end{lemma}

\begin{proof}
The proof is similar to the one in Lemma 3.11 from \cite{MTT}, except that we need to estimate the boundary terms at $\tt=T$ and $\tt=2T$.

To keep the ideas clear we first prove the lemma with $\Box_K$ replaced by $\Box$.
We consider a cutoff function $\chi$ supported in
$[8/20, 22/10]$ which equals $1$ on $[9/20, 21/20]$. Let
$$
\beta(\tt, \rs)=\chi\Big( \frac{\rs}{\tt} \frac{T}{R}\Big)
$$
Note that $\beta\equiv 1$ on $C_T^R$, and that $\beta$ is supported in $\tilde C_{T}^R$.

Integrating $\beta \,\Box w^2\!/2=\beta \big(w\Box w+ m^{\alpha\beta} \pa_\alpha w\,\pa_\beta w\big)$ by  parts twice gives
\begin{equation}
\int_{T}^{2T} \int \beta( |\pa_x w|^2 - |\partial_t w|^2) dx \,dt
=\int_{T}^{2T} \int \Box w \cdot \beta w dx\, dt -\frac12  \int_{T}^{2T}\int (\Box \beta) w^2 dx\, dt
-\int\big(\beta w \partial_t w-\beta_t w^2\!/2\big)\,  dx \Big|_{T}^{2T}.
\label{inp}\end{equation}
Since we can write $w_t=(Sw-x^i\pa_i w)/t$ it follows after integration by parts that
\begin{equation}
\int\beta w \partial_t w\,  dx=\frac{1}{\tt}\int\beta w S w\,  dx
+\frac{1}{2 \tt}\int w^2\pa_i (x^i \beta)\,  dx.
\end{equation}
Since $|\pa_i (x^i \beta)|+\tt|\pa_t\beta|\leq C$ on the support of $\beta$, it follows that the boundary terms are bounded by
\begin{equation}
CT^{-1}\Big(  \|w(2T,\cdot)\|^2_{L^2(\tilde C_{T}^R (2T))} + \|S w(2T,\cdot)\|^2_{L^2(\tilde C_{T}^R (2T))}+ \|w(T,\cdot)\|^2_{{L^2(\tilde C_{T}^R (T))}} + \|S w(T,\cdot)\|^2_{L^2(\tilde C_{T}^R (T))}\Big).
\end{equation}

Let $\chi(t/T)$ be another smooth cutoff such that $\chi(2)=1$ and $\chi(1)=0$. We write
\begin{equation}
w(2T,x)^2=\int_{1/2}^1 \frac{d}{ds} (\chi w^2)(s2T,sx)\, ds
= \int_{1/2}^1 S(\chi w^2)(s2T,sx)\, \frac{ds}{s}
=  \int_{T}^{2T} S(\chi w^2)(t,t x/2T)\, \frac{dt}{t}
\end{equation}
and thus
\begin{equation}
\|w(2T,\cdot)\|_{L^2(\tilde C_{T}^R (2T))}^2
\lesssim \frac{1}{T} \| S(\chi w^2)(t,x)\|^2_{L^2(\tilde C_{T}^R)} \lesssim \frac{1}{T}  \Big(\|w \|^2_{L^2(\tilde C_{T}^R)} + \|S w \|^2_{L^2(\tilde C_{T}^R)}\Big).
\end{equation}
 A similar argument holds for $2T$ replaced by $T$, and for $w$ replaced by $S w$. Hence the boundary term can be estimated by
\begin{equation}
 \frac{1}{T^2} \sum_{j=0}^2 \|S^j w \|^2_{L^2(\tilde C_{T}^R)} .
\end{equation}

To estimate $\pa w$ we use the pointwise inequality
\begin{equation}\label{eq:LagrangeEst}
|\pa w|^2 \le M \frac{1}{(\tt-\rs)^{2}} |Sw|^2 + \frac{\tt}{\tt-\rs}
(|\pa_x w|^2 - |\partial_t w|^2)
\end{equation}
which is valid inside the cone $C$ for a fixed large $M$.
Hence
\beq\label{LagrangeEst2}
\int \beta|\pa w|^2 dx dt \lesssim \int \frac{1}{(\tt-\rs)^{2}} \beta |Sw|^2
 +  \frac{\tt}{\tt-\rs}|\Box \beta| w^2  + \frac{\tt}{\tt-\rs}
 \beta |\Box w| |w| dxdt
\eq
where all weights have a fixed size in the support of $\beta$.
The function $\beta$ also satisfies $|\Box \beta| \lesssim R^{-2}$.
Then the conclusion of the lemma follows by applying Cauchy-Schwarz
to the last term.

The argument for $C_{T}^U$ is similar. We now consider
$$
\beta(\tt, \rs)=\chi\Big( \frac{\tt-\rs}{\tt} \frac{T}{U}\Big)
$$

We multiply by $\beta w$ and integrate by parts as above. The boundary terms are now controlled by
\begin{equation}
CU^{-1}\Big(  \|w(2T,\cdot)\|^2_{L^2(\tilde C_{T}^U (2T))} + \|S w(2T,\cdot)\|^2_{L^2(\tilde C_{T}^U (2T))}+ \|w(T,\cdot)\|^2_{{L^2(\tilde C_{T}^U (T))}} + \|S w(T,\cdot)\|^2_{L^2(\tilde C_{T}^U (T))}\Big).
\end{equation}
which in turn is controlled, by using the scaling $S$ as above, by
\begin{equation}
 \frac{1}{TU} \sum_{j=0}^2 \|S^j w \|^2_{L^2(\tilde C_{T}^R)} .
\end{equation}

The estimate now follows from \eqref{LagrangeEst2}, using the fact that $ |\Box \beta| \lesssim T^{-1}U^{-1}$.

Now consider the above proof but with $\Box$ replaced by $\Box_K$.
Integrating $\beta \,\Box_K w^2\!/2=\beta \big(w\Box_K w+ g_K^{\alpha\beta} \pa_\alpha w\,\pa_\beta w\big)$ by  parts twice gives
\begin{equation}
-\int_{T}^{2T}\!\!\! \int \beta g_K^{\alpha\beta} \pa_\alpha w\,\pa_\beta w \sqrt{|g_K|} dx \,dt
=\int_{T}^{2T}\!\!\! \int\big( \beta w \Box_K w -\frac12  (\Box_K \beta) w^2\big) \sqrt{|g_K|} dx \,dt
-\frac{1}{2} \int\big(\beta  g_K^{0\alpha}\pa_\alpha w^2 - g_K^{\alpha 0} w^2\pa_\alpha\beta\big)\, \sqrt{|g_K|}  dx \Big|_{T}^{2T}.
\end{equation}
First we estimate the boundary term. The terms with $\alpha=0$ are handled as before and so is the second term with $\alpha>0$. For the first term with $\alpha>0$ we integrate by parts and see that it is bounded by a term of the same form as the second term plus a term of the form
\begin{equation}
\frac{1}{2} \int \beta \pa_\alpha (g_K^{0\alpha}  \sqrt{|g_K|})\, w^2 \, dx\lesssim \int \beta r^{-2} w^2\, dx ,
\end{equation}
which can be estimated as above. To estimate the interior term we just note that
\begin{equation}
\sqrt{|g_K|} g_K^{\alpha\beta} \pa_\alpha w\,\pa_\beta w
= |\pa_x w|^2 - |\partial_t w|^2+O(r^{-1})|\pa w|^2,
\end{equation}
where the error term can be absorbed in the left of \eqref{eq:LagrangeEst} for large enough $R$.

This finishes the proof of \eqref{ctrdu}. \eqref{ctudu} follows in a similar manner.
\end{proof}

Applying Lemma~\ref{lem:L2derL2} to $w_{\alpha}$ for some multiindex $\alpha$, and using \eqref{comm} we obtain the higher order version of the estimates:
\beq\label{L2Rder}
\|\pa w_{\alpha}\|_{L^2(C_{T}^{R})} \lesssim R^{-1}
 \| w_{|\alpha|+n}\|_{L^2(\tilde C_{T}^R)}
  + R \| (\Box_K w)_{|\alpha|+n}\|_{L^2(\tilde C_{T}^R)}
\eq
\beq\label{L2Uder}
\|\pa w_{\alpha}\|_{L^2(C_{T}^{U})} \lesssim U^{-1}
 \| w_{|\alpha|+n}\|_{L^2(\tilde C_{T}^U)}
  + T \| (\Box_K w)_{|\alpha|+n}\|_{L^2(\tilde C_{T}^U)}
\eq

Combining the two estimates above \eqref{L2Rder} and \eqref{L2Uder} with the Sobolev embeddings from Lemma~\ref{l:l2tolinf} and the pointwise estimate for second order derivatives in  Lemma \ref{KS1} we obtain
\begin{corr} \label{l:LinftyderL2}
 For  all $T \geq 1$ and $1 \leq R,U \leq T/4$  we have for some $n$ independent of $\alpha$:
 \begin{equation}
  \!  \|\pa w_{\alpha}\|_{L^\infty(C_T^{R})} \lesssim
\frac{1}{R}  \|w_{\leq |\alpha|+n}\|_{L^\infty(\tilde  C_T^{R})} + R \|  (\Box_K w)_{|\alpha|+n}\|_{L^\infty(\tilde C_{T}^R)},
    \end{equation}
respectively
\begin{equation}
  \!  \|\pa w_{\alpha}\|_{L^\infty(C_T^{U})} \lesssim
\frac{1}{U} \|w_{\leq |\alpha|+n}\|_{L^\infty(\tilde  C_T^{U})} + T \|  (\Box_K w)_{|\alpha|+n}\|_{L^\infty(\tilde C_{T}^U)}.
    \end{equation}
\end{corr}

Finally, we will derive a sharp estimate for the bad first order derivative, following
\cite{L90}.
\begin{lemma}\label{lem:transversalder} Let $D_t=\{x;\, 0\leq t-|x| \leq  t/4 \}$,
$C^q_t=\{x;\, t-|x| =q \}$, and let $\overline{w}(q)$ be any positive continuous function,
where $q=t-r$. Suppose that $\Box\phi=F$. Then the following holds in $D_t$, $t\geq 1$:
\begin{multline}\label{eq:transversalder}
t \,|\pa\phi(t,x)\, \overline{w}(q)| \lesssim\!\sup_{4q \leq
\tau\leq t}
\Big(\|\,q  \, \pa \phi(\tau,\cdot)\, \overline{w}\|_{L^\infty(C^q_\tau)}
+{\sum}_{|I|\leq 1} \| Z^I\phi(\tau,\cdot)\, \overline{w}\|_{L^\infty(C^q_\tau)}\Big)\\
+ \int_{4q}^t\Big( \la\tau\ra\|
\,F(\tau,\cdot)\, \overline{w}\|_{L^\infty(C^q_\tau)} +{\sum}_{|I|+|J|\leq 2} \la\tau\ra^{-1}
\| \pa^I \Omega^J \phi(\tau,\cdot)\, \overline{w}\|_{L^\infty(C^q_\tau)}\Big)\, d\tau.
\end{multline}
\end{lemma}
\begin{proof}

 We write
\begin{equation*}
\Box \phi
= - \frac{1}{r}\,\pa_{v}\pa_{u}(r\phi)
+\frac{1}{r^2}\triangle_\omega \phi,
\end{equation*}
where  $\pa_{u}=\pa_t-\pa_r $ and $\pa_{v}=\pa_{t}+\pa_r$.
Hence in $D_t$
\beq\label{boxest}
\Big|\,\pa_{v}\pa_{u}(r\phi)  \Big|
\lesssim \Bigl|r\Box\phi\Bigr| + {\la r\ra}^{-1}\,\, {\sum}_{|I|+|J|\leq 2} \,\,|\pa^I \Omega^{J} \phi| \lesssim \Bigl|\la t\ra\Box\phi\Bigr| + {\la t\ra}^{-1}\,\, {\sum}_{|I|+|J|\leq 2} \,\,|\pa^I \Omega^{J} \phi|
\eq

Integrating this along the flow lines of the vector field $\pa_{v}$ from the boundary of $D=\cup_{\tau\geq 0}D_\tau$
to any point inside $D_t$ for $t\geq 1$. Using that $\overline{w}$ is constant along the flow lines, and \eqref{boxest}, we obtain
\beq
|\pa_u(r \phi(t,x))\, \overline{w}(q)| \lesssim |\pa_u(r \phi)(4q,3q)\overline{w}(q)| + \int_{4q}^t \Big( \la\tau\ra\|
\,F(\tau,\cdot)\, \overline{w}\|_{L^\infty(C^q_\tau)} +{\sum}_{|I|+|J|\leq 2} \la\tau\ra^{-1}
\| \pa^I\Omega^J \phi(\tau,\cdot)\, \overline{w}\|_{L^\infty(C^q_\tau)}\Big)\, d\tau.
\eq
Moreover
\[
t |\pa_u \phi(t,x) \overline{w}(q)| \lesssim |\pa_u(r \phi(t,x))\, \overline{w}(q)| + |\phi(t,x)\overline{w}(q)|,
\]
and
\[
|\pa_u(r \phi)(4q,3q)\overline{w}(q)| \lesssim |q \pa_u \phi(4q, 3q)\overline{w}(q)| + |\phi(4q, 3q)\overline{w}(q)|.
\]

The last three inequalities yield
\begin{multline}
t |\pa_u \phi(t,x) \overline{w}(q)| \lesssim \!\sup_{4q \leq
\tau\leq t}
\Big(\|\,q  \, \pa \phi(\tau,\cdot)\, \overline{w}\|_{L^\infty(C^q_\tau)}
+ \| \phi(\tau,\cdot)\, \overline{w}\|_{L^\infty(C^q_\tau)}\Big)\\
+ \int_{4q}^t\Big( \la\tau\ra\|
\,F(\tau,\cdot)\, \overline{w}\|_{L^\infty(C^q_\tau)} +{\sum}_{|I|+|J|\leq 2} \la\tau\ra^{-1}
\| \pa^I \Omega^J \phi(\tau,\cdot)\, \overline{w}\|_{L^\infty(C^q_\tau)}\Big)\, d\tau.
\end{multline}

The lemma follows from also using that $r |\pa \phi|\lesssim |r\pa_q \phi| +|S\phi|+|\Omega\phi|$.
\end{proof}

\section{Pointwise estimates from the Minkowski fundamental solution}\label{sec:sharpdecay}

In this section, we translate pointwise bounds on the inhomogeneous terms into pointwise bounds for the solution by using the fundamental solution of the Minkowski metric.

For any $\beta, \gamma, \eta\in\R$, we define the weighted $L^\infty$ norms
\beq\label{Linftywght}
\|G\|_{L_{\beta,\gamma,\eta}^{\infty}} = \| \la r\ra^{\beta} \la t\ra^{\gamma} \la t-r\ra^{\eta}  H(t, r)\|_{L_{t,r}^{\infty}}, \quad H(t,r)= \sum_0^2 \|\Omega^i G (t, r\omega)\|_{L^2(\S^2)}.
\eq

We use the following lemma (see Section 6 of \cite{Toh}).

\begin{lemma}\label{Minkdcy}
 Let $\psi$ solve
\beq\label{Mink1}
\Box \psi = G , \qquad \psi(0) = 0, \quad \pa_t \psi(0) = 0,
\eq
where $G$ is supported in $\{|x| \leq t+R_0\}$. Assume also that $2\leq\beta\leq 3$ and $\eta\geq -1/2$.  We define, for any arbitrary $\delta>0$,
\[
\tilde\eta = \left\{ \begin{array}{cc} \eta-\delta -2& \eta<1 ,  \cr -1 & \eta > 1
  \end{array} \right. .
\]

i) If $\gamma\geq 0$, we have

\beq\label{lindcy1}
r \psi(t, x)\lesssim \frac{1}{\la t-r\ra^{\beta+\gamma+\tilde\eta-1}} \|G\|_{L_{\beta,\gamma,\eta}^{\infty}},
\eq

ii) If $\gamma<0$, we have
\beq\label{lindcy2}
r \psi(t, x)\lesssim \frac{t^{-\gamma}}{\la t-r\ra^{\beta+\tilde\eta-1}} \|G\|_{L_{\beta,\gamma,\eta}^{\infty}},
\eq

iii) If $|x|\leq t-1$, and $\eta>1$, we have
\beq\label{lindcy3}
r\psi(t, x)\lesssim \ln \frac{t}{t-r}\|G\|_{L_{2, 0 , \eta}^{\infty}},
\eq
\end{lemma}

\begin{proof}

Note first that, after a translation in time, we may assume that $R_0=0$.

We use the ideas from \cite{MTT}.  Define
\begin{equation}\label{Hdef}
  H(t,r)= \sum_0^2 \|\Omega^i G (t, r\omega)\|_{L^2(\S^2)}.
\end{equation}
By Sobolev embeddings on the sphere, we have $|G|\lesssim H$. Let $v$ be the radial solution to
\begin{equation}\label{1dbox}
  \Box v = H, \qquad v[0]=0.
\end{equation}

By the positivity of the fundamental solution, we have that $|\psi| \lesssim |v|$. On the other hand, we can write $v$ explicitly:
\begin{equation}\label{Hsol}
rv(t,r) = \frac12 \int_{D_{tr}} \rho H(s,\rho) ds d\rho,
\end{equation}
where $D_{tr}$ is the rectangle
\[
D_{tr}=\{ 0 \leq s - \rho \leq t-r, \quad  t-r \leq s+\rho \leq t+r
 \}.
\]

We partition the set $D_{tr}$ into a double dyadic manner
as
\[
D_{tr} = \bigcup_{R \leq t}  D_{tr}^R, \quad D_{tr}^R = D_{tr}\cap \{R<r<2R\}
\]
and estimate the corresponding parts of the above integral.

We clearly have
\[
\int_{D_{tr}^R} \rho H ds d\rho \lesssim \|G\|_{L_{\beta, \gamma,\eta}^{\infty}} \int_{D_{tr}^R} \rho^{1-\beta} \la s\ra^{-\gamma} \la s-\rho\ra^{-\eta} d\rho ds.
\]

We now consider two cases:

(i) $R < (t-r)/8$. Here we have $\rho \sim R$ and $s\approx s-\rho \approx \la t-r\ra$;
therefore we obtain

\[
\int_{D_{tr}^R} \rho^{1-\beta} \la s\ra^{-\gamma} \la s-\rho\ra^{-\eta} d\rho ds  \lesssim  R^{3-\beta} \la t-r\ra^{-\gamma-\eta},
\]
and after summation, using that $\beta\leq 3$, we obtain
\beq\label{ptwsecpt}
\sum_{R < (t-r)/8} \int_{D_{tr}^R} \rho H ds d\rho \lesssim \frac{\ln\la t-r\ra \la t-r\ra^{3-\beta} }{\la t-r\ra^{\gamma+\eta}} \lesssim \frac{1}{\la t-r\ra^{\beta+\tilde\eta}},
\eq
which is the desired bound in all cases.

(ii) $(t-r)/ 8 < R < t$. Here we have $\rho \sim R$ and $t\geq s\gtrsim R$. Denote $u=s-\rho$.

Assume first that $\gamma\geq 0$; then
\[
\int_{D_{tr}^R} \rho^{1-\beta} \la s\ra^{-\gamma} \la s-\rho\ra^{-\eta} d\rho ds  \lesssim R^{2-\beta-\gamma} \int_{0}^{t-r} \la u\ra^{-\eta} du \lesssim R^{2-\beta-\gamma} \la t-r\ra^{\mu(\eta)},
\]
where
\[
\mu(\eta)= \left\{ \begin{array}{cc} 1-\eta& \eta<1 ,  \cr 0 & \eta > 1
  \end{array} \right. .
\]

If $\beta+\gamma>2$, we obtain after summation
\beq\label{ptwsefar}
\sum_{R > (t-r)/8} \int_{D_{tr}^R} \rho H ds d\rho \lesssim \la t-r\ra^{2-\beta-\gamma+\mu(\eta)},
\eq
which is \eqref{lindcy1}.

On the other hand, if $\beta=2$ and $\gamma=0$, and taking into account that there are $\ln \frac{t}{t-r}$ dyadic regions when $(t-r)/ 8 < R < t$, we obtain \eqref{lindcy3} after summation.

Finally, if $\gamma< 0$ we obtain
\[
\int_{D_{tr}^R} \rho^{1-\beta} \la s\ra^{-\gamma} \la s-\rho\ra^{-\eta} d\rho ds  \lesssim R^{2-\beta} t^{-\gamma} \int_{0}^{t-r} \la u\ra^{-\eta} du \lesssim R^{2-\beta} t^{-\gamma} \la t-r\ra^{\mu(\eta)},
\]

Since $\beta\geq 2$, we obtain after summation
\beq\label{ptwsefar2}
\sum_{R > (t-r)/8} \int_{D_{tr}^R} \rho H ds d\rho \lesssim t^{-\gamma} \la t-r\ra^{2-\beta+\mu(\eta)},
\eq
which is \eqref{lindcy2}.
\end{proof}

\section{Setup for pointwise estimates}

In this section, we will slightly adjust $\Box_K$ to an operator closer to $\Box$ (with respect to the $(\tt, x)$ coordinates). Indeed, we let
\[
P = |g_K|^{1/4}(-g_K^{\tt\tt})^{-1/2}\Box_K (-g_K^{\tt\tt})^{-1/2} |g_K|^{-1/4}.
\]

$P$ is self-adjoint with respect to $d\tt dx$. More importantly, a quick computation yields that
\[
P = \pa_{\alpha} \left(g_K^{\alpha\beta} (-g_K^{\tt\tt}) \partial_{\beta}\right) + V, \quad V = |g_K|^{1/4}(-g_K^{\tt\tt})^{-1/2} \Box_K \left((-g_K^{\tt\tt})^{-1/2} |g_K|^{-1/4}\right).
\]
It is easy to see that $V\in S^Z(r^{-3})$. 

Let us first consider the Schwarzschild metric. In this case we have that for large $r$, $-g_S^{\tt\tt} = g_S^{r^*r^*}$ and $g_S^{\tt r^*} = 0$. We thus have
\[
P = \Box + P_{lr},
\]
where the long range spherically symmetric part $P_{lr}$ has the form
\beq\label{Plr}
P_{lr} = g_{lr}(r)\Delta_{\omega} + V, \qquad g_{lr} \in S^Z(r^{-3}), \qquad V \in S^Z(r^{-3}).
\eq

For the Kerr metric, we use the fact that the metric coefficients have the following properties:
\beq\label{KSdiff}
g_K^{\alpha\beta} - g_S^{\alpha\beta} \in S^Z(r^{-2}),
\eq
\beq\label{Kder}
\pa g_K \in S^Z(r^{-2}), \quad \pa^2 g_K \in S^Z(r^{-3})
\eq
Using \eqref{Plr} and \eqref{KSdiff} we see that we can write
\beq\label{Pdec}
P = \Box + P_{lr} + P_{sr},
\eq
where the short-range part $P_{sr}$ has the form
\beq\label{Psr}
P_{sr} = \partial_\alpha g_{sr}^{\alpha \beta}\partial_\beta, \quad g_{sr}^{\alpha \beta} \in S^Z(r^{-2}).
\eq

Using \eqref{Kder} we see that for any function $\phi$ we have
\beq\label{Peq0}
P \phi = (-g_K^{\tt\tt}) \Box_K\phi + h_1 \phi + h_2  \pa\phi, \quad h_1\in S^Z(r^{-3}), \quad h_2 \in S^Z(r^{-2}).
\eq

Now pick any multiindex $\alpha$. After commuting with vector fields, using \eqref{Pdec}, \eqref{Plr}, and \eqref{Psr}, we obtain
\beq\label{Peq}
P \phi_{\alpha} \in S^Z(1) (\Box_K \phi)_{\leq|\alpha|} + S^Z(r^{-3}) \phi_{\leq |\alpha|+6} + S^Z(r^{-2}) \pa \phi_{\leq |\alpha|+5},
\eq
which in turn implies, using \eqref{Pdec}
\beq\label{Boxeq}
\Box \phi_{\alpha} \in S^Z(1) (\Box_K \phi)_{\leq|\alpha|} + S^Z(r^{-3}) \phi_{\leq |\alpha|+6} + S^Z(r^{-2}) \pa \phi_{\leq |\alpha|+5}.
\eq

Moreover, by finite speed of propagation, and the assumption on the support of the initial data, the right hand side is supported in the forward light cone $\{|x| < \tt + R_0\}$.

We will use \eqref{Boxeq} in the next section to extract more decay for the solution.

\section{The bootstrap argument for the Einstein model}

We now prove Theorem~\ref{mainthm} by using a bootstrap argument. We first write
\begin{equation}\label{indata}
\mathcal{E}_N(0) = \mu_N \epsilon
\end{equation}
where $\mu_N>0$ is a fixed, small $N$-dependent constant to be determined below (see \eqref{ptwseapbdsbadLE}, \eqref{ptwseapbdsgoodLE}).

Let $N_1=\frac{N}2$. We will assume that the following a-priori bounds hold for some large constant $\tilde C$ independent of $\epsilon$ and $\tt$, and a fixed small $\delta>0$
\begin{equation}\label{enapbds}
\mathcal{E}_N(\tt) \leq \tilde C \mu_N\epsilon \la \tt\ra^{\delta},
\end{equation}
\begin{equation}\label{ptwseapbds}
|\phi_{\leq N_1+2}| \leq \frac{\epsilon \rs^{\delta}}{\la \tt\ra}, \qquad |\partial \phi_{\leq N_1+2}| \leq \frac{\epsilon}{\rs^{1-\delta}\la \tt-\rs\ra}
\end{equation}
\begin{equation}\label{ptwseapbdsLL}
|(\partial \phi_{TU})_{\leq N_1+2}| \leq \frac{\epsilon}{\la \tt\ra} 
\end{equation}

Clearly \eqref{enapbds}, \eqref{ptwseapbds} and \eqref{ptwseapbdsLL} hold for small times. We assume now that the bounds hold on some time interval $0\leq \tt\leq T$, and we improve the constants by $1/2$. By the continuity method this implies that the solution exists globally, and that the bounds also hold globally.

In order to improve \eqref{enapbds}, we show that, for small enough $\epsilon$, there is $C_N$ independent of $T$ so that
\beq\label{eneps}
\mathcal{E}_N(\tt) \leq C_N \la \tt\ra^{C_N \epsilon} \mathcal{E}_N(0), \quad 0\leq \tt\leq T
\eq

If we now additionally take $\tilde C = 2C_N$ and $\epsilon < \frac{\delta}{C_N}$ we thus improve the a-priori bound for $\mathcal{E}_N(\tt)$ to
\[
\mathcal{E}_N(\tt) \leq \frac12 \tilde C  \mu_N \epsilon \la \tt\ra^{\delta}.
\]

In order to improve the pointwise bounds, we will show that, for some fixed positive integer $m$, independent of $N$, we have
\begin{equation}\label{ptwseapbdsbadLE}
|\phi_{\leq N-m}| \lesssim \frac{\mathcal{E}_{N}(0)}{\la \tt-\rs\ra^{\delta}\la \tt\ra^{1-\delta}}, \qquad |\partial \phi_{\leq N-m}| \lesssim \frac{\tt^{\delta}\mathcal{E}_{N}(0)}{r\la \tt-\rs\ra^{1+\delta}}
\end{equation}
\beq\label{ptwseapbdsgoodLE}
|(\partial \phi_{TU})_{\leq N-m}| \lesssim \frac{\mathcal{E}_{N}(0)}{r \la \tt-\rs\ra^{1-\delta}} 
\eq

We can now pick a small $\mu_N$ to improve \eqref{ptwseapbds} and \eqref{ptwseapbdsLL}.

\subsection{The energy estimates}
We will now use assumptions \eqref{ptwseapbds} and \eqref{ptwseapbdsLL} to show \eqref{eneps} for small enough $\epsilon$.

By Gronwall's inequality and \eqref{wle}, it is enough to show that
\beq\label{eninhom}
 \|(\Box_K\phi)_{\leq N}\|_{LE^*_w[0, \tt]} \lesssim \int_0^{\tt} \frac{\epsilon}{\tau} \mathcal{E}_N(\tau) d\tau + \epsilon\mathcal{E}_N(\tt)
\eq

We can write, using \eqref{Pass}, \eqref{weaknull} and \eqref{Qass}:
\[
\Box_K \phi \in S^Z(1) (\pa\phi_{TU})^2 + S^Z(1) \pa\phi \tpa\phi
\]

After commuting with vector fields, and using \eqref{hghnull}, we also get that
\beq\label{sharp}
(\Box_K \phi)_{\leq N} \lesssim (\pa\phi_{TU})_{\leq N_1} (\partial \phi_{TU})_{\leq N} + \pa\phi_{\leq N_1} \tpa\phi_{\leq N} + \tpa\phi_{\leq N_1} \pa\phi_{\leq N} + r^{-1} \pa\phi_{\leq N_1} \pa\phi_{\leq N-1}
\eq

The first term is easy. By \eqref{ptwseapbdsLL} we have
\[
\|(\pa\phi_{TU})_{\leq N_1} (\partial \phi_{TU})_{\leq N} \|_{L^1[0,\tt]L^2} \lesssim \int_0^{\tt} \frac{\epsilon}{\tau} \mathcal{E}_N(\tau) d\tau.
\]

Similarly, the last term can be estimated in $L^1L^2$. Indeed, we note that \eqref{ptwseapbds} implies that
\[
|r^{-1} \pa\phi_{\leq N_1}| \lesssim \frac{\epsilon}{\tt}
\]
and thus
\[
\|r^{-1} \pa\phi_{\leq N_1} \pa\phi_{\leq N-1} \|_{L^1[0,\tt]L^2} \lesssim \int_0^{\tt} \frac{\epsilon}{\tau} \mathcal{E}_N(\tau) d\tau.
\]

For the second term, we divide it into two parts. When $r<R_1$ we have by \eqref{ptwseapbds}:
\[
\|\pa\phi_{\leq N_1} \tpa\phi_{\leq N}\|_{L^1[0,\tt]L^2(r<R_1)} \lesssim \int_0^{\tt} \frac{\epsilon}{\tau} \mathcal{E}_N(\tau) d\tau
\]

When $r>R_1$, we use \eqref{ptwseapbds} and the last term in \eqref{energy}:
\[
\|\pa\phi_{\leq N_1} \tpa\phi_{\leq N}\|^2_{LE^*[0,t]} \lesssim \int_0^{\tt} \int_{r>R_1} \frac{\epsilon^2 \tau^{2\delta}}{r \la \tau-\rs\ra^{2+2\delta}} |\tpa\phi_{\leq N}|^2 dV \lesssim \|\epsilon \la \tau-\rs\ra^{\frac{-1-\delta_1}2} \tpa\phi_{\leq N}\|_{L^2[0, t] L^2(r\geq R_1)}^2 \lesssim \left(\epsilon\mathcal{E}_N(\tt)\right)^2
\]

For the third term, note that \eqref{tgimp} and \eqref{ptwseapbds} imply that
\beq\label{tangptwse}
|\overline{\pa} \phi_{\leq N_1}| \lesssim \frac{\epsilon}{\la \tt\ra}
\eq

Using \eqref{tangptwse} gives
\[
\|\tpa\phi_{\leq N_1} \pa\phi_{\leq N}\|_{L^1[0,\tt]L^2} \lesssim \int_0^{\tt} \frac{\epsilon}{\tau} \mathcal{E}_N(\tau) d\tau
\]

Putting all these together we obtain \eqref{eninhom}.

\subsection{The decay estimates}

We now show that \eqref{ptwseapbdsbadLE} and \eqref{ptwseapbdsgoodLE} hold.

The proof uses an iteration procedure. The most important part here is to obtain pointwise decay rates of $\tt^{-1}$ near the trapped set for all components. We start with a weak decay rate of $\tt^{-1/2+C\epsilon}$ given by the slow growth $\tt^{C\epsilon}$ combined with the results of Section 3. We then use Lemma~\ref{Minkdcy} to improve decay in $r$, followed by Corollary~\ref{l:LinftyderL2} to improve the decay of derivatives. Lemma~\ref{lv:smallr} then allows us to turn the $r$-decay into $\tt$ decay. This yields an improved global decay rate of $\tt^{-1+C\epsilon}$, which is barely not enough. We then use Lemma~\ref{lem:transversalder} to improve the decay of the derivative of the good components $\pa\phi_{TU}$ to $\tt^{-1}$ near the cone. We can now go back to the iteration procedure, and use the improved bounds combined with Lemma~\ref{Minkdcy}, Corollary~\ref{l:LinftyderL2} and Lemma~\ref{lv:smallr} to improve the decay rate of all components to $\tt^{-1}$ away from the cone. This finishes the proof.

Let $N_2 = N-13$. We first note that \eqref{ptdecayu} and \eqref{deBound}, combined with the energy bounds \eqref{eneps}, yield the weak pointwise bounds
\beq\label{1stbd}
|\partial \phi_{\leq N_2}| \lesssim \frac{\la \tt\ra^{C\epsilon}\mathcal{E}_{N}(0)}{r \la \tt-\rs\ra^{1/2}}, \qquad |\phi_{\leq N_2}| \lesssim \frac{\la \tt-\rs\ra^{1/2}\mathcal{E}_{N}(0)}{\la \tt\ra^{1-C\epsilon}}
\eq

We now need to improve the decay of $\phi_{\leq N-m}$ and $\pa\phi_{\leq N-m}$. To that extent, we will use Lemma~\ref{Minkdcy}, followed by Lemma~\ref{lv:smallr} and Corollary~\ref{l:LinftyderL2}.

We cannot apply Lemma~\ref{Minkdcy} directly. On one hand, we have no control on the solution for $r\ll 2M$, and on the other hand, the initial data is not trivial. Instead, let
\[
\chi = \chi_1(\rs) \chi_2(\tt)
\]
Here $\chi_1\equiv 1$ for $\rs\geq R\gg M$ and supported in $\rs\geq R/2$, while $\chi_2\equiv 1$ for $\tt\geq 1$ and supported in $\tt\geq 1/2$.

We now consider $\psi_{\alpha\beta} = \chi \phi_{\alpha\beta}$. Using \eqref{Boxeq}, we see that $\psi$ satisfies the system
\[
\Box (\psi_{\leq n}) = G_n, \quad G_n \in S^Z(r^{-2}) \pa\phi_{\leq n+5} + S^Z(r^{-3}) \phi_{\leq n+6} + S^Z(1) (\pa\phi_{\leq n})^2
\]
with trivial initial data, and $G_n$ supported in the region $r\geq R/2$. Using \eqref{1stbd}, we see that, for all $n\leq N_3 := N_2-12$, we have
\[
G_{n+6} \lesssim \mathcal{E}_{N}(0) \left(\frac{\la\tt\ra^{C\epsilon}}{r^3\la \tt-\rs\ra^{1/2}} + \frac{\la \tt-\rs\ra^{1/2}}{r^3\la\tt\ra^{1-C\epsilon}} + \frac{\la\tt\ra^{C\epsilon}}{r^2\la \tt-\rs\ra}\right)
\]

We now apply Lemma~\ref{Minkdcy}. The first term on the right hand side is controlled by the other two terms. For the second term we use \eqref{lindcy1} with $\beta=3$, $\gamma = 1-C\epsilon$ and $\eta=-1/2$. For the third term, we use \eqref{lindcy1} with $\beta=2$, $\gamma = -C\epsilon$ and $\eta=1-C\epsilon$. We obtain
\beq\label{phiimp2}
|\phi_{\leq N_3}| \lesssim  \frac{\tt^{C\epsilon}}{r} \mathcal{E}_{N}(0).
\eq

We now plug in the bounds \eqref{phiimp2} and \eqref{1stbd} into Corollary~\ref{l:LinftyderL2}. We thus obtain for $N_4= N_3-n$ with $n$ from Corollary~\ref{l:LinftyderL2}:
\[
\!  \|\pa \phi_{N_4}\|_{L^\infty(C_T^{R})} \lesssim
\frac{1}{R} \frac{T^{C\epsilon}}{R} \mathcal{E}_{N}(0) + R \left(\frac{T^{C\epsilon}}{RT^{1/2}} \mathcal{E}_{N}(0)\right)^2 \lesssim \frac{T^{C\epsilon}}{R^2} \mathcal{E}_{N}(0)
\]
\[
\!  \|\pa \phi_{N_4}\|_{L^\infty(C_T^{U})} \lesssim
\frac{1}{U} \frac{T^{C\epsilon}}{R} \mathcal{E}_{N}(0) + T \left(\frac{T^{C\epsilon}}{RU^{1/2}} \mathcal{E}_{N}(0)\right)^2 \lesssim \frac{T^{C\epsilon}}{RU} \mathcal{E}_{N}(0)
\]
The last two inequalities can be written as
\beq\label{derimp2}
|\pa\phi_{\leq N_4}| \lesssim \frac{\tt^{1+C\epsilon}}{r^2 \la \tt-\rs\ra} \mathcal{E}_{N}(0)
\eq

We now use Lemma~\ref{lv:smallr}. Note that \eqref{phiimp2} and \eqref{derimp2} yield
\[
\|\la r\ra \phi_{\leq N_4}\|_{LE^1( C_{T}^{<T/2})} \lesssim T^{1/2+C\epsilon} \mathcal{E}_{N}(0)
\]

Moreover, \eqref{1stbd} implies that
\[
\|(\Box_K\phi)_{\leq N_4} \|_{LE^*( C_{T}^{<T/2})} \lesssim T^{-1/2 + C\epsilon} \mathcal{E}_{N}(0)
\]

The two inequalities above and Lemma~\ref{lv:smallr} with $N_5 = N_4-n$ give us
\[
\| \phi_{\leq N_5} \|_{LE^1(C_{T}^{<T/2})} \lesssim T^{-1/2 + C\epsilon} \mathcal{E}_{N}(0)
\]
which combined with the Sobolev embeddings from Lemma \ref{l:l2tolinf} give for $N_6= N_5-13$:
\beq\label{phiimp3}
|\phi_{\leq N_6}| \lesssim \tt^{-1+ C\epsilon} \mathcal{E}_{N}(0)
\eq

We now plug in the bounds \eqref{phiimp3} and \eqref{1stbd} into Corollary~\ref{l:LinftyderL2}. We thus obtain for $N_7= N_6-n$
\[
\!  \|\pa \phi_{\leq N_7}\|_{L^\infty(C_T^{R})} \lesssim
\frac{1}{R} \frac{T^{C\epsilon}}{T} \mathcal{E}_{N}(0) + R \left(\frac{T^{C\epsilon}}{RT^{1/2}} \mathcal{E}_{N}(0)\right)^2 \lesssim \frac{T^{C\epsilon}}{RT} \mathcal{E}_{N}(0)
\]
which combined with \eqref{derimp2} gives
\beq\label{derimp3}
|\pa\phi_{\leq N_7}| \lesssim \frac{\tt^{C\epsilon}}{r \la \tt-\rs\ra} \mathcal{E}_{N}(0)
\eq
This finishes the proof of \eqref{ptwseapbdsbadLE} when $\rs>\tt/2$.

Note also that \eqref{tgimp}, \eqref{phiimp3} and \eqref{derimp3} give
\beq\label{tderimp3}
|\tpa\phi_{\leq N_7-2}| \lesssim \frac{\tt^{C\epsilon}}{r \la \tt\ra} \mathcal{E}_{N}(0)
\eq

We now use the fact that $\psi_{TU}$ actually satisfies better decay estimates. Indeed, note first that
\[
\Box (T^{\alpha} U^{\beta}\phi_{\alpha\beta}) - T^{\alpha} U^{\beta} \Box \phi_{\alpha\beta} \in S^Z(r^{-2}) \phi_{\leq 1}
\]

Using \eqref{Pdec} and \eqref{Peq0} we obtain
\[
\Box \phi_{TU} \in S^Z(1) (\Box_K \phi)_{TU} + S^Z(r^{-2}) \phi_{\leq 6}
\]
and since
\[
(\Box_K \phi)_{TU} \in S^Z(1) \pa\phi\tpa\phi
\]
we thus have
\[
\Box \phi_{TU} \in S^Z(1) \pa\phi\tpa\phi + S^Z(r^{-2}) \phi_{\leq 6}
\]

After commuting with vector fields (in particular using \eqref{hghnull}) and applying the cutoff we thus obtain
\[
\Box (\psi_{TU})_{\leq m} = H_m, \quad H_m \in S^Z(r^{-2}) \phi_{\leq m+6} + S^Z(1) \pa\phi_{\leq m}\tpa\phi_{\leq m} + S^Z(r^{-1}) (\pa\phi_{\leq m})^2
\]

Using \eqref{phiimp3}, \eqref{derimp3} and \eqref{tderimp3}, we see that
\beq\label{Hm}
H_m \lesssim \frac{\mathcal{E}_{N}(0) }{r^2\la \tt\ra^{1-C\epsilon}}, \quad m\leq N_7-2
\eq

Let $N_8=N_7-6$. We now apply Lemma~\ref{lem:transversalder} with $\overline{w}(q) = \la q\ra^{1-\delta}$ to $(\psi_{TU})_{\leq N_8}$. Note first that, due to \eqref{derimp3} and \eqref{phiimp3} we have
\[
\sup_{4q\leq
\tau\leq \tt}
\Big(\|\,q  \, \pa \phi_{\leq N_8}(\tau,\cdot)\, \overline{w}\|_{L^\infty(C^q_\tau)}
+{\sum}_{|I|\leq 1} \| Z^I\phi_{\leq N_8}(\tau,\cdot)\, \overline{w}\|_{L^\infty(C^q_\tau)}\Big) \lesssim \mathcal{E}_{N}(0)
\]

Moreover, \eqref{phiimp3} implies that
\[
\int_{4q}^{\tt} {\sum}_{|I|\leq 2} \la\tau\ra^{-1}
\| \Omega^I \phi_{\leq N_8}(\tau,\cdot)\, \overline{w}\|_{L^\infty(C^q_\tau)}\, d\tau \lesssim \int_{4q}^{\tt} \la\tau\ra^{-1} \frac{\la q\ra^{1-\delta} \mathcal{E}_{N}(0)}{\la \tau\ra^{1-C\epsilon}} d\tau \lesssim \mathcal{E}_{N}(0).
\]
Finally, we obtain by \eqref{Hm} that
\[
\int_{4q}^{\tt} \la\tau\ra\|
\,H_m (\tau,\cdot)\, \overline{w}\|_{L^\infty(C^q_\tau)} \lesssim \int_{4q}^{\tt} \la\tau\ra \frac{\la q\ra^{1-\delta} \mathcal{E}_{N}(0) }{\la \tau\ra^{3-C\epsilon}} d\tau \lesssim \mathcal{E}_{N}(0).
\]

Lemma~\ref{lem:transversalder} thus implies, in conjunction with \eqref{derimp3}, that
\beq\label{derimp4}
|\pa (\psi_{TU})_{\leq N_8}| \lesssim \frac{\mathcal{E}_{N}(0)}{r\la \tt-\rs\ra^{1-\delta}}.
\eq

This finishes the proof of \eqref{ptwseapbdsgoodLE}.

Finally, to obtain a decay rate of $1/{\tt}$ in the interior, we see that, using \eqref{Boxeq} and \eqref{sharp}, we can write our system as
\[
\Box (\psi_{\leq m}) = J_m, \quad J_m \in S^Z(r^{-2}) \pa\phi_{\leq m+5} + S^Z(r^{-3}) \phi_{\leq m+6} + S^Z(1) (\pa\phi_{TU})^2_{\leq m} + S^Z(1) \pa\phi_{\leq m}\tpa\phi_{\leq m} + S^Z(r^{-1}) (\pa\phi_{\leq m})^2
\]
and $J_m$ is supported in the region $\{\tt\geq 1/2, \rs\geq R/2\}$. Due to the improved bounds \eqref{phiimp3}, \eqref{derimp3} and \eqref{derimp4} we obtain
\[
J_{m+6} \lesssim \mathcal{E}_{N}(0) \left(\frac{t^{C\epsilon}}{r^3\la \tt-\rs\ra}+ \frac{1}{r^2\la \tt-\rs\ra^{2-2\delta}} \right), \quad m\leq N_9 := N_8-8.
\]

We now apply Lemma~\ref{Minkdcy} and in particular \eqref{lindcy2} to control the last term. We obtain
\beq\label{phiimp5}
\phi_{\leq N_9} \lesssim \frac{1}{r} \mathcal{E}_{N}(0), \quad \rs<3\tt/4.
\eq

Corollary~\ref{l:LinftyderL2} thus implies, with $N_{10}= N_9-n$:
\beq\label{derimp5}
\pa\phi_{\leq N_{10}} \lesssim \frac{1}{r^2} \mathcal{E}_{N}(0), \quad \rs<3\tt/4.
\eq

We now use Lemma~\ref{lv:smallr}. Note that \eqref{phiimp5} and \eqref{derimp5} yield
\[
\|\la r\ra \phi_{\leq N_{10}}\|_{LE^1( C_{T}^{<T/2})} \lesssim T^{1/2} \mathcal{E}_{N}(0)
\]

Moreover, \eqref{derimp3} implies that
\[
\|(\Box_K \phi)_{\leq N_{10}} \|_{LE^*( C_{T}^{<T/2})} \lesssim T^{-1/2} \mathcal{E}_{N}(0)
\]

The two inequalities above and Lemma~\ref{lv:smallr} give us for $N_{11}= N_{10}-n$:
\[
\| \phi_{\leq N_{11}} \|_{LE^1(C_{T}^{<T/2})} \lesssim T^{-1/2} \mathcal{E}_{N}(0)
\]
which combined with the Sobolev embeddings from Lemma \ref{l:l2tolinf} give with $N_{12}= N_{11}-13$
\beq\label{tphiimp5}
|\phi_{\leq N_{12}}| \lesssim \frac{\mathcal{E}_{N}(0)}{\la \tt\ra}, \quad \rs\leq \tt/2
\eq

Finally, one last application of Corollary~\ref{l:LinftyderL2} with $N_{13}= N_{12}- n$ gives
\beq\label{tderimp5}
|\pa\phi_{\leq N_{13}}| \lesssim \frac{\mathcal{E}_{N}(0)}{r \la \tt\ra}, \quad \rs\leq \tt/2
\eq

This finishes the proof of \eqref{ptwseapbdsbadLE} if we pick $N$ large enough so that $N_{13}\geq N_1$.

\section*{Acknowledgments}
H.L. was supported in part by NSF grant DMS-1500925 and Simons Collaboration Grant 638955. M.T. was supported in part by the Simons collaboration Grant 586051. We would also like to thank the Mittag Leffler Institute for their hospitality during the Fall 2019 program in Geometry and Relativity.


\begin{thebibliography}{1}

\bibitem{Al} S. Alinhac: \textit {An example of blowup at infinity for a quasilinear wave equation}, Asterisque 284 (2003), 1--91.

\bibitem{Al2} S. Alinhac: \textit{On the Morawetz-Keel-Smith-Sogge inequality for the wave equation
on a curved background}. Publ. Res. Inst. Math. Sci. {\bf 42(3)}
(2006), 705--720

\bibitem{AB}
L.~Andersson, P.~Blue: \textit{Hidden symmetries and decay for the wave equation on the Kerr
spacetime}, Annals of Mathematics, Volume \textbf{182} (2015), Issue 3, 787--853

\bibitem{AAG} Y.~Angelopoulos, S.~Aretakis, and D.~Gajic: \textit{Nonlinear scalar perturbations of extremal Reissner-Nordstr\"om spacetimes}. Ann. PDE 6 (2020), no. 2, Paper No. 12, 124 pp.

\bibitem{Ar} S.~Aretakis: \textit{Decay of axisymmetric solutions of the wave equation on extreme Kerr backgrounds}, Journal of Functional Analysis {\bf 263}, no. 9 (2012), 2770--2831.

\bibitem{BS1}
P.~Blue, A.~Soffer: \textit{Semilinear wave equations on the Schwarzschild manifold I: local
decay estimates}. Advances in Differential Equations \textbf{8}
(2003), 595--614.

\bibitem{BS2} P.~Blue, A.~Soffer: \textit{The wave equation on the
    Schwarzschild metric II: Local decay for the spin-2 Regge-Wheeler
    equation}, J. Math. Phys., {\bf 46} (2005), 9pp.

\bibitem{BSter} P.~Blue and J.~Sterbenz: \textit{Uniform decay of local energy and the semi-linear wave equation on Schwarzschild space}, Comm. Math. Phys. \textbf{268} (2006), no. 2, 481--504.

\bibitem{BH} J.-F. Bony and D. H\"afner: \textit{The semilinear wave equation on asymptotically Euclidean manifolds}. Comm. Partial Differential Equations 35 (2010), no. 1, 23--67.

\bibitem{Ch} D. Christodoulou: \textit{Global solutions of nonlinear hyperbolic equations for small initial data}. Comm. Pure Appl. Math., 39(2) (1986), 267--282.

\bibitem{CK2} D. Christodoulou and S. Klainerman: \textit{The global nonlinear stability of the Minkowski space}. Princeton University Press, 1993.

\bibitem{DHRT21} M. Dafermos, G Holzegel, I. Rodnianski and M. Taylor \textit{The nonlinear stability of the Schwarzschild family of black holes} \arXiv{2104.08222}.

\bibitem{DR1} M. Dafermos and I. Rodnianski: \textit{The red-shift effect and radiation decay on black hole spacetimes}.  Comm. Pure Appl. Math. \textbf{62} (2009), no. 7, 859--919.

\bibitem{DR2} M. Dafermos and I. Rodnianski: \textit{A note on energy currents and decay for the wave equation on a Schwarzschild background}.  \arXiv{0710.0171}.

\bibitem{DaRoNotes} M.~Dafermos, I.~Rodnianski: \textit{Lectures on
    black holes and linear waves}, \arXiv{0811.0354}.

\bibitem{DRS16} M.~Dafermos, I.~Rodnianski, and Y.~Shlapentokh-Rothman: \textit{Decay for solutions of the wave equation on Kerr
exterior spacetimes III: The full subextremal case $|a|<M$}, Annals of Math. 183 (2016), no. 3, 787-913.


\bibitem{DP} Y. Deng and F. Pusateri: \textit{On the global behavior of weak null quasilinear wave equations}, Comm. Pure Appl. Math. 73 (2020), no. 5, 1035--1099.

\bibitem{FM} M. Facci and J. Metcalfe: \textit{Global existence for quasilinear wave equations satisfying the null condition}, \arXiv{2208.12659}.


\bibitem{F21} A. Fang \textit{Nonlinear stability of the slowly-rotating Kerr-de Sitter family}, \arXiv{2112.07183}.


\bibitem{F22} A. Fang \textit{Linear stability of the slowly-rotating Kerr-de Sitter family}, \arXiv{2207.07902}.
    
\bibitem{GKS1} E. Giorgi, S. Klainerman and J. Szeftel: \textit{A general formalism for the stability of Kerr}, \arXiv{2002.02740}.

\bibitem{GKS2} E. Giorgi, S. Klainerman and J. Szeftel: \textit{Wave equations estimates and the nonlinear stability of slowly rotating Kerr black holes}, \arXiv{2205.14808}.

\bibitem{HE} S. W. Hawking and G. F. R. Ellis: \textit{The large scale structure of space-time}.  Cambridge Monographs on Mathematical Physics, No. 1. London, New York:  Cambridge University Press 1973.

\bibitem{H} K. Hidano: \textit{The global existence theorem for quasi-linear wave equations with multiple speeds}, Hokkaido Math. J. 33 (3) (2004) 607--636, MR2104831.

\bibitem{HV16} P. Hintz and A. Vasy \textit{The global non-linear stability of the Kerr-de Sitter family of black holes}, Acta Math., 220:1--206 (2018)

\bibitem{HV20} P. Hintz and A. Vasy \textit{Stability of Minkowski space and polyhomogeneity of the metric}, Ann. PDE 6 (2020), no. 1, Paper No. 2, 146 pp.

 \bibitem{IK15}   A. D. Ionescu, S. Klainerman, \textit{On the global stability of the wave-map equation in Kerr Spaces with small angular momentum}. Annals of PDE {\bf 1}(2015), 1-78.


\bibitem{J} F. John:  \textit{Blow-up of solutions of nonlinear wave equations in three space dimensions }, Manuscripta Math. 28 (1979), no. 1-3, 235--268.

\bibitem{K18} Joe Keir: \textit{The weak null condition and global existence using the
  p-weighted energy method}, \arXiv{1808.09982}.

\bibitem{K1} S. Klainerman: \textit{Long time behaviour of solutions to nonlinear wave equations}. In Proceedings
of the International Congress of Mathematicians, Vol. 1, 2 (Warsaw, 1983), pages 1209--1215

\bibitem{K2} S. Klainerman: \textit{The null condition and global existence to nonlinear wave equations}. Nonlinear systems of partial differential equations in applied mathematics, Part 1 (Santa Fe, N.M., 1984), 293--326,
Lectures in Appl. Math., 23, Amer. Math. Soc., Providence, RI, 1986.

\bibitem{KP} S. Klainerman, G. Ponce: \textit{Global, small amplitude solutions to nonlinear evolution equations}.
Comm. Pure Appl. Math. 36 (1983), 133--141.

\bibitem{KS} S. Klainerman and T. Sideris: \textit{On almost global existence for nonrelativistic wave equations in 3D}. Comm. Pure Appl. Math. 49 (1996) 307--321.

\bibitem{KS1} S. Klainerman and J. Szeftel: \textit{Construction of GCM spheres in perturbations of Kerr}, \arXiv{1911.0069}, to appear in Annals of PDE.

\bibitem{KS2} S. Klainerman and J. Szeftel: \textit{Effective results in uniformization and intrinsic GCM
spheres in perturbations of Kerr}, \arXiv{1912.12195}, to appear in Annals of PDE.

\bibitem{KS3} S. Klainerman and J. Szeftel: \textit{Kerr stability for small angular momentum}, \arXiv{2104.11857}.


\bibitem{L90} H. Lindblad, \textit{On the lifespan of solutions of nonlinear wave equations with small initial data}. Comm. Pure Appl. Math. 43 (1990), no. 4, 445--472.


\bibitem{L0} H. Lindblad, \textit{Global solutions of nonlinear wave equations}, Comm. Pure Appl. Math. {\bf 45} (9) (1992), 1063-1096.

\bibitem{L1} H. Lindblad, \textit{Global solutions of quasilinear wave equations}, Amer. J. Math. {\bf 130} (2008), no. 1, 115--157.

\bibitem{L2} H. Lindblad, \textit{On the asymptotic behavior of solutions to the Einstein vacuum equations in wave coordinates}, Comm. Math. Phys. 353 (2017), no. 1, 135--184.

\bibitem{LNS} H. Lindblad, M. Nakamura, C. Sogge,
\textit{ Remarks on global solutions for nonlinear wave equations
under the standard null conditions}, J. Differential Equations 254 (2013), 1396-1436.

\bibitem{LR1} H. Lindblad and I. Rodnianski: \textit{The weak null condition for Einstein's equations}.  C. R. Math. Acad. Sci. Paris 336 (2003), no. 11, 901--906.

\bibitem{LR2} H. Lindblad and I. Rodnianski: \textit{Global Existence for the Einstein Vacuum Equations in Wave Coordinates}. Comm. Math. Phys. 256 (2005), 43--110.

\bibitem{LR3} H. Lindblad and I. Rodnianski: \textit{The global stability of Minkowski space- time in harmonic gauge}. Annals of Mathematics, 171 (2010), 1401--1477.

\bibitem{LS} H. Lindblad and V. Schlue: \textit{Scattering from infinity for semilinear wave equations satisfying the null condition or the weak null condition}, \arXiv{1711.00822}.

\bibitem{LT1}
H. Lindblad, M. Tohaneanu: \textit{Global existence for quasilinear wave equations close to Schwarzschild}, Communications in Partial Differential Equations Volume 43, 2018 - Issue 6, 893--944.

\bibitem{LT2}
H. Lindblad, M. Tohaneanu: \textit{A local energy estimate for wave equations on metrics asymptotically close to Kerr}, Ann. Henri Poincar\'e 21, 2020, no. 11, 3659--3726.


\bibitem{LV22}
O. Lindblad-Petersen, A. Vasy \textit{Wave equations in the Kerr-de Sitter spacetime: the full subextremal range},\arXiv{2112.01355}

\bibitem{Looi} S.-Z. Looi: \textit{Improved decay for quasilinear wave equations close to asymptotically flat spacetimes including black hole spacetimes}, \arXiv{2208.05439}.

\bibitem{LooiT} S.-Z. Looi and M. Tohaneanu: \textit{Global existence and pointwise decay for the null condition}, \arXiv{2204.03626}.

\bibitem{Luk} J. Luk: \textit{The null condition and global existence for nonlinear wave equations on slowly rotating Kerr spacetimes}.
Journal Eur. Math. Soc., {\bf 15} (5) (2013):1629--1700.


\bibitem{MMTT} J. Marzuola, J. Metcalfe, D. Tataru, M. Tohaneanu: \textit{Strichartz estimates on Schwarzschild black hole backgrounds}, Comm. Math. Phys. 293 (2010), no. 1, 37--83.

\bibitem{M22}G. Mavrogiannis \textit{Decay for quasilinear wave equations on cosmological black hole backgrounds (Doctoral thesis)}. https://doi.org/10.17863/CAM.86983

\bibitem{MNS} J. Metcalfe, M. Nakamura, C. Sogge: \textit{Global existence of solutions to multiple speed
systems of quasilinear wave equations in exterior domains}. Forum Math. 17, (2005) 133--168

\bibitem{MS1} J. Metcalfe, C. D. Sogge: \textit{Hyperbolic trapped rays and global existence of quasilinear wave
equations}. Invent. Math. 159 (2005), no. 1, 75--117.

\bibitem{MS2} J. Metcalfe and C. Sogge: \textit{Global existence of null-form wave equations in exterior domains}. Math. Z. 256, (2007) 521--549

\bibitem{MS} J.~Metcalfe and C.~Sogge: \textit{Long-time existence of quasilinear wave equations
exterior to star-shaped obstacles via energy methods}. SIAM J. Math.
Anal. {\bf 38(1)}(2006), 188-209 (electronic)

\bibitem{MTT} J. Metcalfe, D. Tataru, M.~Tohaneanu: \textit{Price's law on nonstationary space-times}, Adv. Math. \textbf{230} (2012), no. 3, 995--1028.

\bibitem{M} C. Morawetz: \textit{Time decay for the nonlinear Klein-Gordon equations}.  Proc. Roy. Soc. Ser. A. {\bf 306} (1968), 291--296.

\bibitem{SW} C. Sogge and C. Wang: \textit{Concerning the wave equation on asymptotically Euclidean manifolds}. J. Anal. Math. 112 (2010) 1--32.

\bibitem{Sh} D. Shen : \textit{Construction of GCM hypersurfaces in perturbations of Kerr},
\arXiv{2205.12336}.

\bibitem{ST} T. Sideris, and S.-Y. Tu: \textit{Global Existence for Systems of Nonlinear Wave Equations in 3D with Multiple Speeds}. SIAM J. Math. Anal. 33 (2001), no. 2, 477--488.

\bibitem{STz} W. Strauss and K. Tsutaya: \textit{Existence and blow up of small amplitude nonlinear waves with a negative potential}. Discr. Cont. Dynamical Systems, 3(2) (1997), 175--188.

\bibitem{TT} D. Tataru and M. Tohaneanu: \textit{Local energy estimate on Kerr black hole backgrounds}, Int. Math. Res. Not. IMRN 2011, no. 2,  248--292.

\bibitem{Toh} M. Tohaneanu: \textit{Pointwise decay for semilinear wave equations on Kerr spacetimes.} Lett. Math. Phys. 112 (2022), no. 1, Paper No. 6, 30 pp.

\bibitem{WY} C. Wang, X. Yu. \textit{Global existence of null-form wave equations on small asymptotically Euclidean manifolds}, Journal of Functional Analysis 266 (2014) 5676--5708.

\bibitem{Y} S. Yang: \textit{Global solutions of nonlinear wave equations in time
dependent inhomogeneous media}, Arch. Ration. Mech. Anal. 209 (2013), no. 2, 683--728.

\bibitem{Yu1} D. Yu: \textit{Modified wave operators for a scalar quasilinear wave equation satisfying the weak null condition}, Comm. Math. Phys. 382 (2021), no. 3, 1961--2013.

\bibitem{Yu2} D. Yu: \textit{Asymptotic completeness for a scalar quasilinear wave equation satisfying the weak null condition}, \arXiv{2105.11573}.






\end{thebibliography}
\end{document}